\documentclass[11pt, twoside, leqno]{elsarticle}

\usepackage{amsmath,amsthm}
\usepackage{amssymb,hyperref}
\usepackage{stmaryrd}
\usepackage{shapepar,bm}
\hypersetup{
    colorlinks = true,
    linkcolor = blue,
    anchorcolor = red,
   citecolor = blue,
    filecolor = red,
    pagecolor = red,
    urlcolor = blue}

\usepackage{enumitem}

\usepackage{graphicx}

\def\V{\mathcal V}

\def\id{\operatorname{id}}

\makeindex

\newcounter{paranum}[section]{}

\newcommand{\etype}[1]{\renewcommand{\labelenumi}{(#1{enumi})}}
\def\eroman{\etype{\roman}}
\def\wbx{{[\mathbf x]}}
\newcommand{\Net}{\mathbb N}

\def\ind{\mathrm{ind}}

\def\tbw{{{\widetilde{\bw}}}}

\def\Vreg{{ V}}
\def\ovDc{\overline{\mathcal D}}
\def\Hom{\operatorname{Hom}}
\newcommand{\trop}[1]{\mathcal{#1}}

\def\adL{{\operatorname{ad}}\,L}
\def\ad{\operatorname{ad}}

\newcommand{\one}{1}
\newcommand{\zero}{0}

\def\Wcal{\mathcal W}

\newcommand{\tT}{\trop{T}}

\def\wb{[\bfb]}
\def\gl{\mathrm{gl}}

\def\tTA{\tT_{\mathcal A}}
\def\NN{\mathbb N}

\def\ovsig{\overline{\sigma}}
\def\sg{\sigma}

\def\bfb{{\mathbf b}}
\def\Bb{{\mathcal B}}

\def\bfv{{\mathbf v}}

\def\Wcal{{\mathcal W}}
\def\Bcal{{\mathcal B}}
\def\Ccal{{\mathcal C}}

\def\d{\partial}

\def\V { V}

\def\bfx{{\mathbf x}}

\def\Sym{{\mathrm{Sym}}}

\def\ZZ{\mathbb Z}

\def\QQ{\mathbb Q}

\def\cocoa{{\hbox{\rm C\kern-.13em o\kern-.07em C\kern-.13em o\kern-.15em A}}}

\def\Dcal{\mathcal D}

\def\Acal{{\mathcal A}}
\def\BSV{\overline{S(\Vreg)}}

\def\End{\mathrm{End}}

\def\blamb{{\bm \lambda}}
\def\bmu{{\bm \mu}}

\def\Pcal{{\mathcal P}}

\def\w2M{\bigwedge^2M}

\def\w{\wedge }
\def\bw{\bigwedge }

\protect
\protect
\protect \protect

\protect
\protect

\def\sra{\rightarrow}
\def\lra{\longrightarrow}

\def\proof{\noindent{\bf Proof.}\,\,}
\def\qed{{\hfill\vrule height4pt width4pt depth0pt}\medskip}
\def\be{\begin{equation}}
\def\ee{\end{equation}}
\def\bclm{\begin{claim}}
\def\eclm{\end{claim}}
\def\beqn{\begin{eqnarray}}
\def\eeqn{\end{eqnarray}}
\def\beqn*{\begin{eqnarray*}}
\def\eeqn*{\end{eqnarray*}}

\newtheorem{theorem}{Theorem}[section]

\newtheorem{lemma}[theorem]{Lemma}

\newtheorem{constr}[theorem]{Construction}
\newtheorem{cond}[theorem]{Conditions}
\newtheorem{lem}[theorem]{Lemma}
\newtheorem{proposition}[theorem]{Proposition}
\newtheorem{prop}[theorem]{Proposition}

\theoremstyle{definition}
\newtheorem{definition}[theorem]{Definition}
\newtheorem{defn}[theorem]{Definition}
\newtheorem{claim}[theorem]{}

\newtheorem{remark}[theorem]{Remark}
\newtheorem{rem}[theorem]{Remark}
\newtheorem{example}[theorem]{Example}

\numberwithin{equation}{section}

\begin{document}


\baselineskip=17pt


\title{Clifford   semialgebras}
\tnotetext[t1]{The research of the second author was supported by Finanziamento di Base della Ricerca,(no. 53\_RBA17GATLET). The research of the third author was supported by Israel Science Foundation Grant No. 1623/16.}
\author{Adam Chapman}
\address{School of Computer Science, Academic College of
Tel-Aviv-Yaffo, Rabenu Yeruham St., P.O.B~8401~Yaffo, 6818211,
Israel} \ead{adam1chapman@yahoo.com}
\author {Letterio Gatto}
\address{Dipartimento di Scienze Matematiche Politecnico di
Torino Corso Duca degli Abruzzi 24, 10129 Torino}
\ead{letterio.gatto@polito.it}
\author{Louis Rowen}
\address{Department of Mathematics, Bar-Ilan University,
Ramat-Gan 52900, Israel} \ead{rowen@math.biu.ac.il}


\begin{abstract}
Continuing the theory of systems, we   introduce a theory of
Clifford semialgebra systems, with application to representation
theory  via Hasse-Schmidt derivations on exterior semialgebras. Our
main result, after the construction of the Clifford  semialgebra, is
a formula describing the exterior semialgebra as a representation of
the Clifford semialgebra, given by the endomorphisms of the first
wedge power.
\end{abstract}
\begin{keyword}
Clifford semialgebras, exterior semialgebras, Schubert derivations,
exterior semialgebra  representation of endomorphisms, bosonic
vertex operator representation of Lie semialgebras of endomorphisms.
\MSC[2010] Primary 15A75; Secondary 17B69, 14M15, 05E05
\end{keyword}
\maketitle
\thanks{$^*$The first author was supported by the Israel Science Foundation grant 1623/16}
\thanks{The authors thank the referee for a careful reading, and for
sound advice on improving the presentation.} 

\tableofcontents

\section{Introduction}
\numberwithin{equation}{section}
The main purpose of this paper is twofold. On one hand we intend to
walk our first steps  towards a representation theory of Lie
semialgebras, in the sense of \cite[Section 8.3]{Row16}, by giving a
closer look to the case of Clifford  semialgebras, introduced in
Section \ref{sec:sec5}, where we work out the first theoretical
basics. Secondly,  we put our abstract picture at work in section
\ref{def:HSD} and \ref{8mnthm}, by providing a proper framework to our Theorem
\ref{thm432}, which computes, in the same spirit of \cite{BeCoGaVi}
and \cite{BeGa}, within the classical algebra framework, the shape of a
generating function that describes the exterior
$\Acal$--semialgebra, introduced and studied in \cite{GaR}, as a
representation of endomorphisms of a free $\Acal $--module.   The
aim is that of  excavating the semialgebraic  roots of classical
mathematical
 phenomenologies, to investigate to which extent certain
 classical results
in the theory of representation of certain infinite dimensional Lie
algebras can be extended to a tropical context, in which one
typically works with modules over semialgebras.

Tropical algebra has been used in a variety of applications; one of
them, relevant to this paper is the use of a tropical exterior
algebra in matroid theory in~\cite{GG}. The approach here is in the
more general framework of ``triples,'' ``surpassing relations'' and
``systems,'' as explained in \cite{JuR1, Row16, Row17}, which
unifies the classical theory and tropical theory and other examples
including hyperfields, as explained in \cite{AGR}, and provides a
framework  for studying linear algebra and module theory \cite{JMR}
in semialgebras, by providing a formal substitution for negation. In
brief, one starts with a semimodule spanned by a distinguished set
of ``tangible'' elements. ``Triples'' are endowed with an abstract
negation map, and ``systems'' are triples endowed with ``surpassing
relations,'' which often replace equality in classical theorems. Our
main result is Theorem \ref{thm432}, in which we propose a more
transparent analog for the ``Fermionic version'' of the generating
function that Date, Jimbo, Kashiwara and Miwa (DJKM) \cite{DJKM}
provided to describe the representation of Lie algebras of matrices
of infinite size having all  but finitely many entries zero.  (See
\cite[Section 5.3]{KRR} for a concise readable account and
\cite{SDIWP} and \cite{BeGa} for generalizations.)

 To be more precise, we explicitly describe the product of elementary
matrices by the basis elements of our exterior semialgebra. To do so
in a uniform way, we use the actions on the basis elements by the
generating functions of the elementary matrices. This procedure
which is classical and widely used, requires  care in the
semialgebra situation, due to many minor verifications needed to
show that many equalities that hold true in the classical and
well-established case can be replaced by {\em surpassing relations}.
We
 put the basis element in a generating function to get a uniform (and not a case by case) description.
See Section~\ref{8mnthm} for more explanation.


Additional abundant sources of motivation come
from other places in well established literature, showing the
 potential of our subject to interact with other parts of
mathematics. One first remark in this sense is  that most  scholars
working in the theory of symmetric functions are nowadays familiar
with the result by Laksov and Thorup \cite{LakTh1} showing that the
ring $S:=\ZZ[e_1,\ldots,e_r]$ of (commutative) polynomials in $r$
indeterminates  is isomorphic, as a $\ZZ$-module, to the $r$-th
exterior power $\bw^r\ZZ[X]$.  Recall that the  former is endowed
with its distinguished basis of  Schur polynomials parametrized by
partitions of length at most $r$. A deep theorem says that the
product of two Schur polynomials is an integral non-negative linear
combination  of Schur polynomials \cite[p.~142]{Macdonald}, i.e. all
the Littlewood-Richardson coefficients are non-negative. Since the
Schur polynomials are parametrized by partitions, one may define the
module $\bw^r\NN[X]$ over the semiring $\NN$, as the $\NN $--module
generated by partitions of length at most~$r$. Therefore the direct
sum $\bw \NN[X]:=\bigoplus_{r\geq 0}\bw^r\NN[X]$ can be thought of
as a prototype of the exterior semialgebra to which we add a further
structure of negation naturally provided by the switch map defined
in degree $\geq 2$, following the philosophy of \cite{JuR1,Row16}.

In the present paper the negation map is extended to degrees $1$ and
$0$ via a trick using operators. In conclusion, the construction of
the exterior semialgebra  alluded to above, besides being
interesting in its own right, is natural and motivated by the
semialgebra of polynomials. This was performed in detail in
\cite{GaR}, giving as an application the extension of the
Cayley-Hamilton theorem in the case of endomorphisms of
semialgebras.

Looking more closely at the semialgebra of endomorphisms of the
exterior semialgebra, it is easy to realize that it naturally
satisfies certain properties (namely, the axioms listed in
Section~\ref{sec:sec5}) which inspired our construction of the
Clifford semialgebras as an abstract framework in which to put our
concrete observations. On one hand this should not be surprising,
because it agrees with the known  classical picture concerning usual
exterior algebras, but on the other hand we were impressed that the
same yoga works, taking care about  some delicate technical detail,
in a more encompassing context. Accordingly, in this paper some
basic representation theory of the Lie semialgebra of the
endomorphisms of the free module $\Acal[x]$ over a semialgebra
$\Acal$ is then proposed. This program could be launched because the
apparatus of Hasse-Schmidt derivations on exterior algebras (as in
\cite{GaSa}) and of the distinguished ones, called {\em Schubert},
can be extended almost verbatim to the case of exterior
semialgebras. The  price we  pay is  in replacing (but nontrivially,
see e.g., Proposition \ref{prop424})  equality by the surpassing
relation. Hasse-Schmidt derivations, in our context, are formal
power series with coefficients in the endomorphisms of the exterior
semialgebra. The latter, in a nutshell,  can be roughly thought of
as the free $\Acal$ module generated by the set $\Pcal$ of all
partitions, additionally endowed with a negation. The reason why the
Schubert derivation $\sigma_+(z)=\sum_{i\geq 0}\sigma_iz^i$ is so
called   is that $\sigma_i\blamb=\sum \bmu $, where $\bmu$ runs over
all the partitions $(\mu_1,\ldots,\mu_r)$ such that
$\mu_1\geq\lambda_1\geq\cdots\geq \mu_r\geq \lambda_r$ and
$|\bmu|=|\blamb|+i$; this is the classical rule of Pieri holding in
Schubert calculus when multiplying cohomology classes in
Grassmannians.

 In \cite{GaSa2} one deals with finite wedge powers of
possibly finite dimensional vector spaces, and the DJKM
representation translates to proclaim that the singular cohomology
of the Grasmannian  is a module over the Lie algebra of  $n\times n$
matrices. We extend this result (and hence, in a sense, also  DJKM)
to the case of the Lie semialgebra of the endomorphisms of
$\Acal[x]$. Most of the equalities  are replaced by surpassing
relations, which are among the key notions of the theory of systems.
Another detail in the context of this paper -- we develop a theory
of power series over the exterior  semialgebra, viewed in terms of
negation maps.

\paragraph{\bf Plan of the Paper}  The plan of the paper is as
follows.
 Section~\ref{prelim} contains the preliminaries concerning the systemic
 (and thus tropical) framework
 developed along the lines of \cite{Row16,Row17} and references therein.
Basic motivating examples are given in Section \ref{motivv}.

Section \ref{sec:sec1}  reviews the construction of the exterior
semialgebra from~\cite{GaR}, and recasts it in terms of the regular
representation (Definition~\ref{regrep}), which utilizes the
endomorphism system
 of (\S\ref{Endsys}). We are able thereby to
show how the tensor functor  takes  modules to
 triples (Theorem~\ref{exth}).

 Section \ref{sec:sec5} is devoted to studying Clifford semialgebras in
 general,
  to provide the appropriate framework for the
 second part of the paper.  The  definition of Clifford
semialgebra in Definition~\ref{Cliff1} is rather concise, and a
generic  construction is proposed in~\ref{stdG}, but, in contrast to
the exterior semialgebra, complications arising when attempting to
introduce the negation map, cf.~Construction~\ref{stdGN}. Thus we
must examine the situation both with and
 without a negation map. The abstract
 construction is given in Construction~
 \ref{stdG},
  whereas the natural model of Clifford semialgebras arising from the
 exterior semialgebra is given already in
 Section~\ref{sec:sec1}.
(In contrast with other applications of systems, we need a
generalization incorporating the ``surpassing relation'' $\preceq$
in our negation map,
 to obtain what here we call $\preceq$-systems.)
 In the
case of polynomial semirings, i.e. $\Acal[x^1,\ldots,x^n]$, where
$\Acal$ is a semialgebra, we obtain the Weyl semialgebra,  the
``symmetric'' counterpart of the Clifford semialgebra seen as a
semialgebra of endomorphisms of the exterior semialgebra. The ``base
square-free'' symmetric  semialgebra (the Giansirancusa tropical
exterior semialgebra of \cite{GG}), an epimorphic image of
$\Acal[x^1,\ldots,x^n]$, is obtained by setting all the squares of
basis elements to zero.


  We then repeat the constructions of Section~\ref{sec:sec5} in the case of the
  base square-free symmetric semialgebra,
  and much of the combinatorial formalism for exterior semialgebras can be mimicked in this case.
  We  describe the semialgebra of  endomorphisms of the tropical Grassmannian,
   and  designate it as the tropical Clifford semialgebra.

In Section \ref{sec:sec6} we study the relevant example of the
exterior semialgebra
  as a regular representation of a canonical Clifford semialgebra, which arises naturally in a concrete context. As shown in the subsequent section, this is also relevant for application to representation theory and mathematical
  physics.


Vertex operators are also defined in this context in
Section~\ref{def:HSD}, and they take a classical familiar form over
a semialgebra $\Acal$ which contains the non-negative rational
numbers. Our culminating result is Theorem \ref{thm432}, our version
of \cite[Theorem 6.4]{GaSa2} or \cite[Theorem 5.9]{BeNa}. It
provides a formula expressing the generating function of the action
of the Lie semialgebra of endomorphisms on fixed degree elements of
the exterior semialgebra by means of the elementary matrices
$E_{ij}$. Clearly this suffices to fully describe the module
structure as each endomorphism is a finite linear combination of the
elementary ones.

\noindent

   \section{Preliminaries and Notation}\label{prelim}

As usual we denote as $\NN$  the monoid of  natural numbers
(including $0$), and $\NN^+:= \NN \setminus \{0\}.$ Throughout
$\Acal$ denotes a fixed commutative associative base semiring
cf.~\cite{golan92} (i.e., satisfies all the axioms of ring except
negatives, e.g.~$\NN$ or $\QQ_{+}$ or the max-plus algebra). Let
$n\in\NN\cup\{\infty\}$. Modules are defined as usual.
$(\tT,\cdot,\one)$ denotes an Abelian multiplicative monoid with
unit element $\one.$

\begin{defn}\label{modu131}
A (left) $\tT$-\textbf{(monoid) module}   is an additive monoid
$(V,+,\zero _V)$ together with scalar multiplication $\tT\times V
\to V$ satisfying distributivity over $\tT$  in the sense that
$$a(v_1+v_2) = av_1 +av_2,\quad  \forall a \in \tT,\ v _i \in V
,$$ also stipulating that $a\cdot\zero_{V} = \zero_{V}$ for all $a$
in $\tT$, $\one_\tT v =v$, and $(a_1a_2)v = a_1 (a_2 v)$, for all
$a_i \in \tT$ and $v\in V$.


%
%
\end{defn}

 $\tT$-\textbf{semialgebras} are
$\tT$-modules which are semirings. $\mathcal A$ will always be a
commutative
 $\tT$-\textbf{semialgebra}.   $\End_{\mathcal A} V$ denotes the semialgebra of
module homomorphisms $V \to V,$ i.e., group homomorphisms $f:V\to V$
satisfying $f(av) = af(v), \forall a\in \mathcal A, v\in V.$ For
notational convenience, we designate $\zero_V \in \End _\mathcal A
V$ for the $0$ homomorphism, i.e., $\zero_V\cdot V = \zero.$ Also
there are maps $\mathcal A \to \End _\mathcal A V$ sending $a\mapsto
l_a,$ the left multiplication map $v\mapsto av,$ and when $V$ is a
semialgebra, $ V\to  \End _\mathcal A  V$ given by $v \mapsto l_v.$

\begin{defn}\label{modu13} The free $\Acal$-module of countable rank is denoted as $$V=\Acal[x]:=\bigoplus_{i\geq 0}\Acal x^i.$$ If $J$ is any countable index set, we
denote the factor module
$$
V_J:={V\over \sum \Acal x^i:i\notin J}
$$
If $J=[n]=\{0\leq i<n\}$, for $n\in\NN\cup \infty$, we write $V_n$
for $\bigoplus_{0\leq i<n}\Acal x^i$, and its base will be denoted
$\bfx^n$.
\end{defn}
Every element of $V$ is a polynomial with coefficients in $\Acal$,
i.e. a finite linear combination of powers of $x$.
 This can
be thought of as $\displaystyle{{\mathcal A[x]} \over {\sum \Acal
x^j: j \ge n}.}$ \footnote{Technically  we are dealing with
congruences, so the notation, which we use  repeatedly, means that
we are modding out the congruence generated by all $(x^j,\zero),\ j
\ge n.$   }
%

\begin{defn}\label{negmap}
 A \textbf{negation map} on  a $\tT$-module $V$
is an injective semigroup homomorphism $(-) :V\to V$ of order $\le
2$, together with a map  $(-) :\tT\to \tT$ also of order $\le 2$,
written $v\mapsto (-)v$, satisfying
\begin{equation}\label{eq:I.1.110} (-)((-)v) = v, \quad
((-)a)v = (-)(av) = a ((-)v),\quad\forall a\in \tT,\quad v \in
V.\end{equation} When $V $ is a semialgebra $\mathcal A,$ we require
\eqref{eq:I.1.110} for all $a ,v \in \mathcal A $.

 We write $v(-)w$ for $v +(-)w,$ and
$v^\circ$ for  $v(-)v$, called a \textbf{quasi-zero}. $V^\circ: = \{
v^\circ: v\in V\}.$  A \textbf{submodule with negation map} is a
submodule $W$ of $V$ satisfying $(-)W = W.$
\end{defn}

\begin{lem}  If a $\mathcal A$-module $V$ has a   negation  map $(-)$,
then: \begin{enumerate}\eroman
 \item $V^\circ $ is a submodule of $V$.
 \item $V\times V$ is a module with  negation, under the
 diagonal action.
 \item $\End _{\mathcal A} \, V$
has a  negation map given by
$$((-)f) (v) =(-)(f(v)).$$\end{enumerate}
\end{lem}
\begin{proof} Easy verifications.
\end{proof}

 \begin{defn}\label{sursys0}
 A \textbf{triple}
is a collection $(\mathcal A, \tTA, (-)),$  where   $\mathcal A$ is
a $\tT$-module and $(-)$ is a negation map on  $\mathcal A$, and  $
\tTA$ is a subset of $\mathcal A$ closed under~$(-)$, satisfying:
\begin{enumerate}\eroman
 \item $\tTA \cap
\mathcal A^ \circ = \emptyset$,

 \item $\tTA \cup \{ \zero\}$
generates $( \mathcal A,+)$ additively.
\end{enumerate}

(We do not require in general that $\tTA$ is a monoid.)

\begin{lem} To verify that a group homomorphism $f:V\to V$ is in
$\End _{\mathcal A} V$ it is enough to check that $(\sum a_i)v =
\sum (a_iv)$ and
 $f(av) = af(v), \forall a_i, a\in \tT, v\in V.$
\end{lem}
\proof  If $b = \sum a_i$ for $a_i \in \tT$ then
$$f(bv) = f(\sum a_i v ) = \sum f(a_i v) = \sum a_i f(v) = b f(v).$$
\qed

 By
\textbf{partial order} we always mean a partial order $\preceq$ on
$\mathcal A$ as a  $\tT$-module, i.e.,  satisfying the following
conditions for all $b,b'\in \mathcal A$:

 \begin{cond}\label{condi}$ $
  \begin{enumerate}
 \eroman
  \item If $b \preceq c$ and $b' \preceq c'$   then  $b + b' \preceq c
   + c'.$
    \item   If  $a \in \tT$ and $b \preceq b'$ then $a b \preceq ab'.$
     \end{enumerate}
 \end{cond}

A \textbf{surpassing relation} is a partial order, also satisfying
\begin{cond}\label{condi1}$ $
  \begin{enumerate}
  \item [(iii)]  $b \preceq b'$
  whenever $b + c^\circ = b'$ for some $c\in \mathcal A$.
    \item  [(iv)] If $b \preceq b'$ then $(-)b \preceq (-)b'$.
 \end{enumerate}
 \end{cond}

(In other words $(-)$ is preserved under $(-)$. If $0 \preceq b'$
then $0 = (-)0 \preceq (-)b'.$)
 We also write $b \succeq b'$ to denote that $b'\preceq b.$

 A \textbf{ system}
$(\mathcal A, \tTA, (-),\preceq)$ is a  triple together with a
surpassing relation, also satisfying:
  \begin{enumerate}
   \item[(v)]    If  $a\preceq a' $ for $a,a' \in \tTA,$ then $a =  a'.$

 \item [(vi)]\textbf{unique negation}: If $a +a' \succeq \zero$ for
$a,a'\in\tTA,$ then $a'
     =(-)a.$
  \end{enumerate}
   \end{defn}

A consequence of  (i) and (iv) is   $c(-) c \succeq \zero. $ We have
(v) in order to have equality on $\tT$ generalizing classical
formulas.

   The surpassing relation is of utmost importance, since it
replaces equality in many classical formulas.

 The most common surpassing relation defined on a triple is
$\preceq _\circ $, given by $b \preceq _\circ b'$
  whenever $b + c^\circ = b'$ for some $c\in \mathcal A$.
In this case $b' \succeq 0$ iff $b' = c^\circ$ for some $ c.$

\begin{lem}\label{makesur}\begin{enumerate}\eroman
\item Conditions~\ref{condi}
 translate to: A partial order   is a subset $S$ of $\mathcal A
\times \mathcal A$ satisfying for all $b,b',c$:

  \begin{enumerate}\eroman
  \item $(b,b)\in S$.
  \item If $(b ,c)\in S$ and $(b' , c')\in S$   then  $(b + b' , c
   + c')\in S$.
    \item   If  $a \in \tT$ and $(b , b')\in S$ then $(a b , ab')\in S$.
    \end{enumerate}

Conditions~\ref{condi1} for surpassing relation translate to:
\begin{enumerate}
  \item  [(d)]   $(b,b + c^\circ) \in S$.
    \item  [(e)] If $(b , b')\in S$ then $((-)b ,(-)b')\in S$.
    \end{enumerate}
 \item Given any subset $S_0 \subseteq \mathcal A \times
\mathcal A$, we  can obtain a partial order by taking $S$ to be the
smallest $\tT$-submodule   of $\mathcal A \times \mathcal A$
 containing $S_0$. We    obtain a  surpassing
relation  by taking $S$ to be the smallest $\tT$-submodule (with
 negation) of $\mathcal A \times \mathcal A$
 containing $S_0 + (\zero \times \mathcal A ^\circ)$.
\end{enumerate}\end{lem}
\proof (i) Conditions~\ref{condi}  translate to these conditions of
(i), under the familiar interpretation of a relation on $\mathcal A$
as a subset $S$ of $\mathcal A \times \mathcal A$.

(ii) The conditions (a),(b) are the definition of submodule. (a),
(d),(e) are those involving the surpassing relation and negation
map. \qed

\begin{defn} \label{Lies2} A \textbf{systemic module} over a semiring with preorder $(\mathcal A,\preceq)$
 is an $\mathcal A$-module~$M$ together with a submodule $\tT_M$ and negation map $(-)$ and surpassing
relation $\preceq$ preserving $\tT_M$ and satisfying $(-)(ay) =
a((-)y)$ as well as $ay \preceq a'y'$  for $a \preceq a' \in
\mathcal
 A$ and $y\preceq y' \in M.$
\end{defn}

(When $\mathcal A$ has a multiplicative unit $\one$ and a negation
map then the negation map on $M$ could be given by $(-)y :=
((-)\one)y$ for $y\in M$).
 In our applications, $M$ often will be an ideal of a semiring  system $\mathcal A$.

\begin{prop}\label{Endsys}
If $M$ is a systemic module over a semiring  $\mathcal A,$ then
taking $\tT_{\End M}$ to be the set of endomorphisms sending $\tT_M$
to $\tT_M$ and $\overline{\End M}$ to be the subalgebra of $\End M$
spanned by $\tT_{\End M}$, $(\overline{\End M}, \tT_{\End M},
(-),\preceq)$ is a system where we define $(-)f: y \mapsto
(-)(f(y))$ and $f\preceq g$ when $f (y) \preceq  g(y)$ for all $y
\in M$.
\end{prop}
\proof Easy verification.
 \qed

$(\overline{\End M}, \tT_{\End M}, (-),\preceq)$ is called the
\textbf{endomorphism system} of $M$.

 Recall from
\cite[Definition~2.37]{JuR1} that the systemic $\mathcal A$-modules
comprise a category, where a morphism, which we will call a
\textbf{$\preceq$-morphism}
$$f: (\mathcal M, \tT_{\mathcal M}, (-), \preceq)\to (\mathcal M',
\tT_{\mathcal M'} , (-)', \preceq')$$ is a map $f: \mathcal M \to
\mathcal M'$ satisfying the following
properties  for $a  \in \Acal$ and   $c,  c_i$  in~$\mathcal M$:
\begin{enumerate}\eroman
\item
$f (\zero ) = \zero  .$
\item $ f((-)c_1)=   (-)
f(c_1);$
\item $ f(c_1 + c_2) \preceq ' f(c_1) + f( c_2) ;$
\item  $ f(a  c)=  a   f( c) $.
\item $ f(c_1) \preceq ' f(c_2)$ if  $c_1\preceq c_2.$
\end{enumerate}

%


\claim{\bf The standard dual base.}\label{dub}
 The  \textit{standard dual base} $ {\bm\d}=(\d^j)_{j\in \NN} \subset \Hom (V,\Acal)$ of $(x^i)_{i\in
 J}$ is defined by
$\d^jx^i=\delta_{ij}$. In the classical case where $\Acal$ contains
the positive rationals, it may be identified with the differential
operator
$$
\d^j:=\left.{1\over j!}{\d\over \d x^j}\right|_{x=0}
$$
which motivates the notation which we use to emphasize  the analogy
with the   Weyl  algebra acting on $\Acal[x]$, generated by the
multiplication by $x$ and the partial derivative $\d^1$ subject to
the relation $x\d ^1=\d ^1 x+\mathrm{id}$.  $V^*:=\bigoplus_{j\geq
0}\Acal \d^j$ is called the \textbf{restricted dual} of $V$.
Define an action $\, \lrcorner \,$ of $\mathcal V ^*$ on $ \mathcal
V $ as:
$$
\d\, \lrcorner \, v=\d(v),\qquad\qquad \forall (\d,v)\in\mathcal V
^*\times\mathcal V .
$$

\section{Main motivating examples}\label{motivv}

Here are the examples to be used throughout this paper.

\claim{\bf The regular representation.} We  present a way of embedding a semialgebra into a system, using
Proposition~\ref{Endsys}.

\begin{definition} \label{regrep}  Suppose $V$ is a module over an $\Acal$-semialgebra $\mathcal A.$ The \textbf{regular representation} $\Psi: \mathcal A  \to  \End_\Acal V$ sends
$b\in \mathcal A$ to the map $\Psi(b): v\mapsto bv.$
\end{definition}

As in usual algebra, this maps $\mathcal A$   into $\End_\Acal V$,
and is an injection when $V$ is faithful over $\mathcal A$, which
will always be the case in this paper.

\begin{lemma}\label{reg2} When $(V,\tT_V,(-),\preceq)$ is a systemic
module, the image of $\mathcal A$   in the system
$(\overline{\End_\Acal V}, \tT_{\End_\Acal V},(-),\preceq)$ is a
subsystem.
\end{lemma}
\proof $(-)\Psi(b):v \to (-)v.$ If $b =\sum_{i=1}^t a_i$ then
$\Psi(b)= \sum_{i=1}^t \Psi (a_i) \in \overline{\End_\Acal V}.$ \qed

  Lemma~\ref{reg2} will be used to put our basic notions (exterior semialgebra, Clifford
  semialgebra) in the context of systems. We turn to the notion of
  Lie semialgebra with a preorder.

\begin{defn} \label{Lies0}  A \textbf{Lie $\preceq$-semialgebra  with a preorder $\preceq$}  is an $\mathcal A$-module $L$ with a negation map $(-)$, endowed with
 an $\mathcal A$-bilinear \textbf{$\preceq$-{Lie bracket}}
$[\phantom{w}\phantom{w}]$
 satisfying, for all $x,y\in L$:
\begin{enumerate}\eroman
 \item $[xy] +[yx]\succeq \zero_L,$

 \item $[x x] \succeq \zero_L,$

\item $\ad_{[x y]} + \ad _y \ad _x \preceq
\ad _x \ad _y , $ where $\ad _x(z) = [xz].$
\end{enumerate}
\end{defn}

It is more intuitive to work with systemic modules.

\begin{defn} \label{Lies} A \textbf{systemic Lie $\preceq$-semialgebra}  is a systemic $\mathcal A$-module $L$ with a negation map $(-)$, endowed with
 an $\mathcal A$-bilinear \textbf{$\preceq$-Lie bracket}
$[\phantom{w}\phantom{w}]$
 satisfying, for all $x,y\in L$:
\begin{enumerate}\eroman
 \item $[xy] = (-)[yx],$

 \item $[x x] \succeq \zero_L,$

\item $\ad_{[x y]} \preceq
\ad _x \ad _y (-) \ad _y \ad _x, $ where $\ad _x(z) = [xz].$
\end{enumerate}
\end{defn}

\begin{lem} Any systemic Lie $\preceq$-semialgebra is a Lie
$\preceq$-semialgebra in the sense of Definition~\ref{Lies0}.
\end{lem}
\begin{proof} $[xy]-[yx]= [xy](-)[xy] \succeq \zero;$ $$\ad_{[x y]}
\preceq \ad_{[x y]} + \ad _y \ad _x (-)\ad _y \ad _x  \preceq \ad _x
\ad _y (-)\ad _y \ad _x .$$
\end{proof}

\begin{lem} [{Jacobi $\preceq$-identity, \cite[Lemma~10.5]{Row16}}]\label{ideal12}
\begin{enumerate}\eroman
\item
 $ [[ b b']v]+[ b'[bv]] \preceq [ b[b'v]]$ for all
$b,b',v\in L$ for a Lie $\preceq$-semialgebra.

\item $\ad
_b(\ad _{b'}(v)) (-) [ b'[bv]] [ b'[bv]]$ for all $b,b',v\in L$ in a
systemic Lie semialgebra.
\end{enumerate}
\end{lem}
\begin{proof} (i) $ [[ b b']v]+[ b'[bv]]= \ad _{[ b b']}(v)+ \ad _{b'}(\ad _{b}(v))\preceq \ad _b(\ad _{b'}(v)) =
[ b[b'v]].$

(ii)  $ [[ b b']v] \preceq [[ b b']v] +[ b'[bv]](-)\preceq \ad
_b(\ad _{b'}(v)) (-) [ b'[bv]] [ b'[bv]].$
\end{proof}

Although (i) in
 Lemma~\ref{ideal12} is put in greater generality, (ii) is the $\preceq$-surpassing version of
  Jacobi's identity.

 In
any semiring with a negation map $ (-) $, we write $[b,b']$ for the
\textbf{Lie commutator} $bb' (-) b'b.$

  \begin{prop}[{\cite[Proposition~10.7]{Row16}}]\label{Pois1} Any associative
  semialgebra $A$ with negation map becomes a Lie
  $\preceq_\circ$-semialgebra, denoted $A^{(-)}$,
  under the Lie product $[bb'] = [b,b']$.\end{prop}

Define $\ad L = \{ \ad _b : b\in L\},$ a Lie sub-semialgebra of
$\End _{\Acal} L$ under the Lie product.
  \begin{prop}[{\cite[Proposition~10.6]{Row16}}] If $L$ is a Lie semialgebra, then there is a Lie
 $\preceq$-morphism
  $\ad: L \to \adL$, given by $b \mapsto \ad_b$ (In fact   $\ad$ respects the Lie product and preserves addition).
\end{prop}

 $  \mathcal V
_n $ is the free module, as in Definition~\ref{modu13}, over
$\mathcal A$ generated by $\bfx^n:=\{ x^j :0\leq j <n\}$. Let \be
S(\mathcal V )={\mathcal A}[x^i : 0\leq i < n] \ee be the symmetric
(i.e., commutative polynomial) semialgebra of $\mathcal V $, \be
\BSV:={S(\mathcal V )\over ({x^i}x^i\sim 0)} ,\ee the base
square-free symmetric semialgebra (which is the Giansirancusa
tropical Grassmannian of \cite{GG}). There is an obvious
$\tT$-semialgebra epimorphism $S(\mathcal V )\lra \BSV$ mapping to
zero all the words in $\bfx$ involving a repetition of a letter.


%
%

%
(This suggests that we could define a more general Weyl semialgebra
in $\End_{\mathcal A}  S(\mathcal V )$ associated to a bilinear
form, cf.~Definition~\ref{bilf}.)


\claim{\bf The tensor system and extended tensor system.} Tensor
products are defined naturally in the semialgebra context,
cf.~\cite{Ka2}. Putting $\V ^{\otimes 0}:=\Acal$, $ \V ^{\otimes
1}:=\V $, and
  $\V ^{\otimes n}:=\V ^{\otimes n-1}\otimes_{\Acal} \V $,
we let
$
  T^k(\V )=    \V  ^{\otimes k},\quad T^{\geq 2}(\V )=\bigoplus_{n\geq 2}\V ^{\otimes n}$. The direct sum
  $$
   T(\V )=\bigoplus_{n\geq 0}\V ^{\otimes n}.
  $$
is called the \textbf{tensor semialgebra}, an important structure
from which we extract a natural negation map.

Motivated by the wish
of getting a working notion of  exterior semialgebra, we define a
negation map on $T^{\geq 2}(\V )$ by $$(-) v_1\otimes v_2 \otimes
\dots =  v_2\otimes v_1 \otimes \dots .$$ Define $\tT^{\geq 2}$ to
be the multiplicative monoid of simple
 tensors at length at least two of basis elements of $\V $,
 i.e., $x^{i_1}\otimes x^{i_2} \otimes \dots$.

\begin{lemma}\label{neg0} As an ${\mathcal A}$-module, $(T^{\geq
2}(\V ),\tT^{\geq 2},(-),\preceq_\circ)$ is a system.
\end{lemma}
\proof In fact $\tT ^{\geq 2}$ is a base of $T^{\geq 2}(\V )$, where
$T^{\geq 2}(\V )^\circ$ is spanned by $( \tT^{\geq 2})^\circ,$ so
unique negation is clear.\qed

This has the flavor of exterior algebras, so we call it the
\textbf{pre-exterior system}, which is a useful tool in
representation theory. But this only works as a module, not as a
semialgebra, since
$$(-)(v_1\otimes v_2)(v_3\otimes v_4) = (v_2\otimes v_1)\otimes
(v_3\otimes v_4),$$ whereas $$(v_1\otimes v_2)((-)(v_3\otimes v_4))
= (v_1\otimes v_2)\otimes (v_4\otimes v_3).$$

To make $(-)$ work for semialgebras we define  \begin{equation}
 \label{ext} \overline{ T}^k(\V )=  T^k(\V )/\mathfrak R,\end{equation}
 where $\mathfrak R$ is the congruence equating $v_1  \otimes \dots \otimes
 v_k$ with  $v_{\pi(1)}  \otimes \dots \otimes
 v_{\pi(k)}$ for every even permutation $\pi\in\Sym_k.$

\begin{lemma}\label{neg3} As an ${\mathcal A}$-semialgebra, $(\bar T^{\geq
2}(\V ),\bar\tT^{\geq 2},(-),\preceq_\circ)$ is a system.
\end{lemma}
\proof We identify $((-)v_1  \otimes \dots \otimes
 v_k)\otimes w_1  \otimes \dots \otimes
 w_\ell = v_2\otimes v_1  \otimes \dots \otimes
 v_k\otimes w_1  \otimes \dots \otimes
 w_\ell $ with $(v_1  \otimes \dots \otimes
 v_k)\otimes (w_2\otimes w_1  \otimes \dots \otimes
 w_\ell) = (v_1  \otimes \dots \otimes
 v_k)\otimes ((-)w_1 \otimes w_2 \otimes \dots \otimes
 w_\ell)$ since the subscripts differ by a product of two
 transpositions, which is an even permutation.\qed

 As seen in \cite{GaR}, the system of
Lemma~\ref{neg3} suffices for many applications, but we would like
to extend this to a system that also contains the degree 1 part. Thus
our pre-exterior system must be enlarged, using the regular
representation.

Consider the  maps:

$$
\id, (-):\Acal \sra \End_\Acal(\bar T^{\geq 2}(\V ))
$$
and
$$
l,r:\V \sra \End_\Acal(\bar T^{\geq 2}(\V ))
$$
 defined by
$$
\id(a)(v_1\otimes v)=a v_1\otimes v\qquad \mathrm{and} \qquad (-)a
(v_1\otimes v)= (-)av_1\otimes v,
$$
$$
l(u)(v_1\otimes v)=u\otimes v_1\otimes v\qquad \mathrm{and} \qquad
r(u) (v_1\otimes v)=v_1\otimes u\otimes v,
$$ for all $(u,v_1\otimes v)\in \V \times \bw(\V )$ such
that $v_1\in\V $.

We shall identify $\V $ with $Im(l)$, and also set $(-)\V :=Im(r)$,
which becomes an $\Acal$--module when one defines $a((-)v)=(-)av\in
(-)V$ for all $(a,u)\in\Acal\times \V $. We write $(-)u$ for $r(u)$,
and $(-)((-)u)=u$.
\begin{definition} \label{defneg}
 The   \textbf{degree zero component} for $\widetilde{ \overline{ T}}(\V )$   is
the $\Acal$-module  $\widetilde {\overline{ T}}^0(\V
):=Im(\id)\bigoplus Im((-)). $

 The   \textbf{degree one component} for
$\widetilde{ \overline{ T}}(\V )$   is the $\Acal$-module
$\widetilde {\overline{ T}}^1(\V ):=Im(l)\bigoplus Im(r). $

 $\widetilde{\overline{ T}}^{\ge 2}(\V ): =   {\overline{ T}^{\ge 2}}(\V
).$
\end{definition}
There is an obvious  $\Acal$--epimorphism $\rho:\widetilde
{\overline{ T}}^1(\V )\sra \V $, given by $l(u)\mapsto u$ and
$r(u)\mapsto u$, so that $\rho^{-1}(u)=\{l(u),r(u)\}$.
\begin{definition}\label{def:defnegmap}
Define $\gl(\widetilde {\overline{ T}}^1(\V ))$ as the set of all
$f\in \End_\Acal(\bar T(\V ))$ such that
$$
 \gl(\widetilde {\overline{T}}^1(\V ))=\left\{\,f\in
\End_\Acal(\bar T(\V
))\,\,\left|\,\,\,\begin{matrix}((-)f)(l(u))&=&r(f(u))\cr\cr
((-)f)(r(u))&=&l(f(u))\end{matrix}\right.\,\,\,\right\}
$$
i.e., the set of all $\Acal$-linear maps  $f$ such that $(-)f$
switches $r$ and $l$.
\end{definition}

\begin{proposition}
If $f,g\in \gl(\widetilde {\overline{ T}}^1(\V ))$, then
$$
(f\circ (-)g)(u)=(-)f(g(u))
$$
i.e.
$$
(f\circ (-)g)(u)\w v=v\w f(g(u))\qquad v\in\V
$$

\end{proposition}
\proof Let $u:=l(u)\in \tilde T_1(\V )$. Then
$$
f(-g)(u)=f((-)g(u))=(-)f(g(u)).
$$
\qed

The    map  $ (-)v:\V \lra \End_\Acal\overline{ T}(\V )$ is given by
$$
(-)v=\underline{\phantom{w}}\otimes v: \overline{ T}(\V )\sra
\overline{ T}(\V )
$$
 mapping $u_1\otimes u$ (for $u_1\in \V $ and $u\in  \overline{ T}(\V )$)
 to
$$
((-)v)(u_1\otimes u)=(u_1\otimes v)\otimes u.
$$

We extend this to all degrees $\ge 1$.
\begin{remark} The regular representation $\Psi_{\overline{ T}^{\ge 1}(\V )}:\overline{ T}(\V )\to  \End_\Acal\overline{ T}(\V )$
(see Definition~\ref{regrep})  sends $v_1 \otimes \dots \otimes v_m$
to the map $$v_1' \otimes \dots \otimes v_k' \mapsto v_1 \otimes
\dots \otimes v_m \otimes v_1' \otimes \dots \otimes v_k'.$$
\end{remark}

\begin{lemma}\label{neg1}
 The regular representation $\Psi_{\overline{ T}^{\ge 1}(\V )}$ is a
homogeneous $\Acal$-injection of semialgebras, of degree $1$, i.e.,
$v\otimes \underline{\phantom{w}}\,:\overline{ T}^{i}(\V )\sra
\overline{ T}^{i+1}(\V )$. Furthermore, the negation map on
$\overline{ T}^{\geq 2}(\V )$ induces a  negation map on
$\Psi_{\overline{ T}(\V )}( \overline{ T}(\V ))$ by
$$(-)\Psi_{\overline{ T}(\V )}(v_1 \otimes v_2 \otimes \dots \otimes v_m)( v_1' \otimes \dots \otimes
v_k') =  v_1 \otimes v_2 \otimes \dots \otimes v_m \otimes v_2'
\otimes v_1'\otimes\dots \otimes v_k',$$ which can be combined with
our previous negation map on $\widetilde {\overline{ T}}^1(\V ).$
\end{lemma}
\proof The analog of the usual regular representation also holds for
semialgebras. The  map switching the order of the tensors serves as
the negation map throughout.

More precisely, we  extend the negation map defined on $\overline{
T}^{\geq 2}(\V )$ to a negation map
$$
(-):\widetilde{\overline{ T}}(\V )\sra \widetilde{\overline{ T}}(\V
),
$$
by defining $(-)a$ to be the element of $\widetilde{\overline{
T}}^{0}$ such that
$$
((-)a) v_1\w\cdots\w v_k=a v_2\w v_1\w \cdots\w v_k,
$$
and defining $(-)v$ to be the element of $\widetilde{\overline{
T}}^{1}$ such that
$$
((-)v)\w v_1\w\cdots\w v_k=v_1\w v\w v_2\w\cdots\w v_k.
$$

Notice that
$$
(-)v\w w=v\w (-)w,
$$
so the negation of $\widetilde{\overline{ T}}^1(V)$ is compatible
with that of $\overline{ T}^{\geq 2}(\V )$.\qed

The upshot of this is:

\begin{proposition}\label{neg70}  The data $(\widetilde{ \overline{ T}}(\V ), \widetilde{\overline{\tT}}(\V ),(-),\preceq_\circ)$ is a system.
\end{proposition}

\begin{remark} $\widetilde{\overline{ T}}^0(V)$ is just the direct sum of two copies of $\Acal.$ In our
applications, it is superfluous, since we take tensors of length
$\ge 1$.
\end{remark}

\section{Exterior systems}\label{sec:sec1}

 In this section we refine the construction of the exterior
semialgebra in \cite[Definition~3.1]{GaR}, first to obtain a system,
but also because we want later to construct a Clifford semialgebra
as its semialgebra of endomorphisms. We start with  two
constructions from \cite{GaR}.

\claim{\bf Exterior semialgebras of types $1$ and $2$.} We define

$$\bw^k \V  := \displaystyle{\widetilde{\overline{ T}}^k(\V )\over (x^i\otimes x^i\sim
  0,\, i\in I)}; \qquad \bigwedge^{\ge 2}\V  = \oplus _{k\ge 2} \bw^k \V ;  \qquad \bw\V  := \oplus _{k\ge 0} \bw^k \V .$$
\begin{defn}
The exterior semialgebra of \textbf{{\bf type 1}} is $ \bw
(\Acal[x]), $ considered with respect to the juxtaposition product
on the $\bw^k\V $.
\end{defn}

\begin{defn}
The exterior semialgebra of \textbf{{\bf type 2}} is:
$$
 {\overline{ \bw\V }}:{\widetilde{\overline{ T}}(\V )\over
{(v^{\otimes 2}\sim 0, v \in \Acal[x])}}.
$$
\end{defn}

\begin{remark}
The switch map on $\overline{ T}^ {\ge 2}(\V )$ induces  a negation
map  $(-)$ on $\overline{ T}^{\ge 2}(\V )$ given by
 the formula
 \begin{equation} \label{sw1} (-) v_{1}\otimes v_{2}\otimes\ldots
 \otimes
v_{k} = v_{2}\otimes v_{1}\otimes\ldots \otimes v_{k},
\end{equation}
which combined with the negation maps on $\widetilde{\overline{
T}}^0(\V )$ and $\widetilde{\overline{ T}}^1(\V )$ provides a
negation map on $\widetilde{\overline{ T}}(\V )$.

Together with \eqref{ext} we have
 \begin{equation} \label{sw1}
 v_{\pi(1)}\otimes\cdots\otimes v_{\pi(k)}= (-)^\pi v_{1}\otimes
 v_{2}\otimes\ldots \otimes
v_{k},
\end{equation}
\end{remark}

\begin{remark} In type 1 we are not requiring the elements
of $\V $ to be square 0; for instance $(x^i+x^j)^{\otimes
2}=x^i\otimes x^j + x^j\otimes x^i$. The ``exterior'' structure is
built in via the switch map.
\end{remark}

In both types we denote the product by $\w$, abusing notation. In
particular
$$
(-)(u\w v)=((-)u)\otimes v=u\w (-)v=v\w u,\qquad \forall (u,v)\in \V
^2.$$

\begin{remark}\label{exth0}
  $( \bw^{\geq 2}\V , \tT^{\ge 2}(\V ), (-), \preceq_\circ)$
is a system. It differs from the system  of   Lemma~\ref{neg3} only
insofar as we have modded out the $v\otimes v$, but the notation
remains the same.\end{remark}

\claim{\bf Construction of the extended exterior system.}  The
system of Remark~\ref{exth0} suffices for the applications in
 \cite{GaR}, but we would like to have a system that also includes
the degree 1 part of the exterior semialgebra. To attain this, our
exterior system will be slightly bigger than the exterior
$\Acal$-algebra. We repeat the argument used in constructing $\tilde
T,$ modding out $(x^i\otimes x^i\sim
  0,\, i\in I)$ for type 1, and   modding out $(v\otimes v\sim
  0,\, v \in V)$ for type 2.

\begin{lemma}\label{neg17}
 The regular representation $\Psi_{\bw^{\ge 1}\V }$ is a
homogeneous $\Acal$-injection of semialgebras of degree $1$, i.e.,
$v\otimes \underline{\phantom{w}}\,:\bigwedge^{i}\V \sra
\bigwedge^{i+1}\V $. Furthermore, as in Lemma~\ref{reg2}, the
negation map on $\bw^{\geq 2}(\V )$ induces a   negation map on
$\Psi_{\bw \V }( \bw^{\geq 2}\V )$ by
$$(-)\Psi (v_1 \otimes v_2 \otimes \dots \otimes v_m)( v_1' \otimes \dots \otimes
v_k') =  v_2 \otimes v_1 \otimes \dots \otimes v_m \otimes v_1'
\otimes \dots \otimes v_k',$$ which can be combined with our
previous negation map on $\widetilde{\bw}^1\V .$
\end{lemma}

We write $\widetilde{\bw}^k \V$ for $\Psi_{\bigwedge^{k} \V}$. Then
we have the negation maps $(-)\Psi_\alpha (v_1 \otimes v_2 \otimes
\dots) = v_2 \otimes v_1 \otimes \dots $ for $\alpha \in
\bigwedge^{0}\V = \Acal$ and
 $(-)\Psi_v (v_1 \otimes v_2 \otimes \dots) = v \otimes v_2 \otimes
v_1 \otimes \dots. $

\begin{definition} The \textbf{extended exterior} $\Acal$-semialgebra of type $1$ is
\be \left(\widetilde\bw\V , \w\right):={\big(\widetilde{\bw}^0\V
\,\oplus \big(\widetilde{\bw}^1\V \,\oplus \bw^{\geq 2}(\V
)\big)\over (x^i\otimes x^i\sim 0)}, \ee all viewed inside
$\End_\Acal(\bw^{\geq 2}\V )$    via the regular representation.
\end{definition}
%
%
%
%
%

(Note that in the classical algebraic situation,
$(-)\widetilde{\bw^1}\V  = \widetilde{\bw^1}\V .$) We identify
$\bw^{k}(\V )$ with $\widetilde{\bw^{k}}(\V )$ for $k\ge 2,$ so
$\tT_{\ge 2}(\V )$ is viewed inside  $\widetilde{\bw}\V .$
\begin{definition}
Let  $\tT (\bw\V )= \bw^1\V  \cup (-)\bw^1\V  \cup \tT_{\ge 2}(\V )$
with respect to the \textbf{juxtaposition product} that we denote
classically by  $\wedge$. The \textbf{extended exterior system}
(type 1) is
$$\left(
\widetilde\bw\V ,\tT (\bw\V ),(-),\preceq_\circ\right)$$

 The \textbf{extended exterior semialgebra}
of type 2, with its system, is
 defined analogously, with \be
\left(\overline{\widetilde\bw\V },\w)
\right)={\big(\widetilde{\bw}^0\V \,\oplus \big(\widetilde{\bw}^1\V
\,\oplus \bw^{\geq 2}(\V )\big))\over (v\otimes v\sim \zero)}, \quad
\qquad \forall v\in \V . \ee
\end{definition}

\begin{remark}
Both cases were studied in \cite{GaR} without the extension of
degree 1, and although we use the notation of Case I, the same
observations hold for Case~II). In fact, there is a
$\preceq$-epimorphism from $ \widetilde\bw\V $ to
$\overline{\widetilde \bw\V }$, sending every quasi-zero to $\zero$.
\end{remark}

Putting everything together, we have

\begin{theorem}\label{exth} For any semiring $F$, there is a functor from
$F$-modules to exterior systems, sending  $V$ to $(\widetilde\bw\V ,
\tT(\bw\V ), (-), \preceq_\circ)$,
\end{theorem}
\noindent
which is the main statement enabling us to view the exterior theory in terms of systems.

\claim{\bf Partitions.} We can reinterpret the extended exterior semialgebra combinatorially
via Young Diagrams. Let
$$
\Pcal_{r,n}:=\{(\lambda_1\geq\cdots\geq
\lambda_r)\in\NN^r\quad|\quad \lambda_1\leq n+1-r\}
$$
 be the set of all partitions (typically denoted as $\blamb$) with at most  $r$ components, whose Young diagrams are
 contained in an $r\times (n-r)$ rectangle.
We write $\Pcal_r$ for~$\Pcal_{r,\infty}$ and $\Pcal$ for
$\Pcal_\infty$. The \textbf{weight} of a partition is $|\blamb|=\sum
\lambda_i$.
 The partition $\blamb'$ \textbf{conjugate} to $\blamb$
is the partition whose Young diagram is the transpose of that
of~$\blamb$. For example if $\blamb=(3,3,2,1,1)$ then
$\blamb'=(5,3,2)$.

Let $\Acal$ be a semialgebra, $n\in\NN^+\cup\{\infty\}$, and $0\leq
r<n$.
 Let $n\in\NN^+\cup\{\infty\}$, and $\V_n$ be the free $\Acal$-module of rank $n$ with
 base $\bfx^n:=\{x^i: \ {0\leq i<n}\}$ (cf.~Definition~\ref{modu13}),
 i.e.~$ \V_n=\bigoplus_{0\leq i<n}\Acal  x^i.$
\begin{definition} Let $0\leq r<n$. The \textbf{extended exterior $\Acal$-module} $E^r(\V_n)$
is defined as follows: $E^0(\V_n):=\Acal$, $E^1(\V
_n):=\Acal^{(n)},$ and for $r \geq 2$,
$$
E^r(\Vreg_n):=\bigoplus_{\blamb\in\Pcal_{r,n}}\Acal \blamb,
$$
the  free $\Acal$-module generated by $ \Pcal_{r,n}$.
\end{definition}
Following  traditional notation we   write
$$
E^r(\Vreg_n):=\bigoplus_{\blamb\in\Pcal_{r,n}}\Acal \cdot \wbx^r_\blamb,
$$
where $ \wbx^1_\blamb=x^{\lambda_1}$ and for $r\geq 2$,
 \be
 \wbx^r_\blamb=x^{r-1+\lambda_1}\w
x^{r-2+\lambda_2}\w\cdots\w   x^{\lambda_r}, \ee although this
notation is  redundant since $r$ is understood from the partition.
By convention $E^r(\Vreg_n)=0$ if $r\ge n$. (This follows from  $x^i
\bw x^i = 0.$) The total \textbf{extended exterior $\Acal$--module}
is, by definition, the free $\Acal$-module of rank~$2^n$:
$$
E(\Vreg_n)_=\bigoplus_{r\geq 0}E^r(\Vreg_n)
$$
 Let $\{e_1,\ldots,e_r)$ be
commuting indeterminates over $\Acal$.
\begin{proposition}
There is an $\Acal$-module epimorphism
$$
\Bcal_r(\Acal):=\Acal[e_1,\ldots, e_r]\sra E^r(\Vreg_n),
$$
given in the proof, which is an isomorphism if and only if
$n=\infty$.
\end{proposition}
\proof The natural base of $\Bcal_r(\Acal)$ is given by all the
monomials $e_1^{i_1}\cdots e_r^{i_r}$. Let us denote by $\blamb\in\Pcal_r$  the partition with at most $r$ parts conjugated to
$(1^{i_1}\cdots r^{i_r})$. Therefore
the claimed epimorphism is given by
$$
e_1^{i_1}\cdots e_r^{i_r}=e^\blamb\mapsto \wbx^r_\blamb
$$
and the image is zero  if and only if $r\geq n$.\qed

\claim{\bf The action of the symmetric group.}
Let ${\mathbb X}^r \subset \NN^r$ be the set of all $r$-tuples of distinct natural numbers, acted  on by the subgroup $\Sym _r^+$ of even permutations  of $\Sym _r$ on $\{ 0,1,\dots ,r-1\}$
 (whose action on $\mathbb X$ is switching the first two components). Let $\rho_r:=(r-1,r-2,\ldots,0)\in {\mathbb X}^r$ (we write $\rho$
 when $r$ is understood), and fix the  transposition $\tau =(1 2)\in \Sym _r$.
Then ${\mathbb X}^r$   decomposes into the disjoint union
 \begin{equation} \label{ch}
{\mathbb X}^r=\, \Sym _r^+\cdot (\rho +\Pcal_r)  \,\cup    \,  \tau
(\Sym _r^+\cdot (\rho +\Pcal_r)),
\end{equation}
which also defines a
 map
$$
\Lambda:{\mathbb X}^r\sra \Pcal_r
$$
given by $$\Lambda(i_1,\ldots, i_r)=
\Lambda(\tau(i_1,\ldots,i_r))=\blamb, \quad \forall
(i_1,\ldots,i_r)\in \Sym _r^+\cdot (\rho+\blamb).$$ We  also have a
map ${\mathbb X}^r/\Sym _r^+\sra \{\rho,\tau\rho\}$  sending
$(\rho+\blamb)\mapsto \rho,$  $\ \tau(\rho+\blamb)\mapsto \tau\rho$.

\claim{\bf The combinatorial construction.}\label{combcon} We define
a negation map on $\mathbb X^r$ by putting
$$(-)(i_1,\ldots,i_r) = \tau(i_1,\ldots,i_r), $$ noting that $ \tau(i_1,\ldots,i_r)\Sym _r^+=
(i_2,i_1,\ldots,i_r)\Sym _r^+,$ and define $(-)E^r(
V_n)$ to
be
$$\{v_2 \wedge v_1 \wedge \dots \wedge v_r  : v_1 \wedge v_2 \wedge \dots \wedge
v_r \in E^r(\mathcal V_n)\}.$$

For all $r\geq 2$ we define
$$
\widetilde{\bigwedge}^{r}V_n:= E^r(V_n)(-)E^r(\Vreg_n)
$$
where $\wbx^r_{(-)\blamb}=(-)\wbx^r_\blamb:=x^{r-2+\lambda_2}\w
x^{r-1+\lambda_1}\w\cdots\w x^{\lambda_r}$. It is an algebra with
respect to juxtaposition
$$
\wbx^r_\blamb\w
\wbx^s_\bmu=\wbx^{r+s}_{\Lambda(\rho_r+\blamb,\rho_s+\bmu)}.
$$
The map $(-): \widetilde{\bw}^r\Vreg_n\sra \widetilde{\bw}^r\Vreg_n $
interchanging the sheets over $E_n^r(\Vreg_n)$ is the negation map
of $\widetilde{\bw}^r\V_n$. Define
$$
\widetilde{\bigwedge}^{\geq 2}\Vreg_n:=\bigoplus_{r\geq
2}\widetilde{\bw}^r\Vreg_n.
$$

\begin{proposition}
One can define two  maps $\Vreg_n\times \Vreg_n \sra
\widetilde{\bw}^{\ge 2}\V $, which are respective $\Acal$-linear
extensions of the set-theoretic maps $(x^i,x^j)\mapsto x^i\w x^j$
and $(x^i,x^j)\mapsto (-)x^i\w x^j=x^j\w x^i$.
\end{proposition}
\proof Clear.\qed

We view  our subsequent examples in this notation.
\section{Clifford semialgebras}\label{sec:sec5}
As in the classical case,  Clifford semialgebras are obtained by generalizing exterior semialgebras, using the
 $\preceq$-version of a quadratic form   motivated by \cite{IKR}. Moreover, the semialgebra of endomorphisms of an exterior semialgebra provides a natural example of Clifford semialgebra in our sense.

\claim{\bf $\preceq$-Bilinear and quadratic forms.} We begin by recalling the following:

\begin{defn}[{cf.~\cite[Definition~6.1]{IzhakianKnebuschRowen2010LinearAlg}}]\label{bilf}
A ${\preceq}$-\textbf{bilinear form} on a  systemic module $(V,
\tT_{V}, (-),\preceq)$ over a semiring ${\mathcal A}$ is a map
$$\Bb: V \times V \to {\mathcal A}$$ which satisfies
\begin{enumerate}
\item $\Bb(a_0 v_0 + a_1 v_1, w) \succeq a_0 \Bb(v_0,w) + a_1
\Bb(v_1, w), $
\item $\Bb(v, a_0 w_0 + a_1 w_1) \succeq a_0 \Bb(v,w_0) +
a_1 \Bb(v, w_1),$
\end{enumerate}
$\forall v, v_i, w , w_i \in V.$

 A  ${\preceq}$-bilinear form $\Bb:
V\times V\to {\mathcal A}$ is \textbf{symmetric} if $\Bb(v,w) =
\Bb(w,v)$ for all~$v,w\in V.$
\end{defn}
\begin{example}
\begin{enumerate}\eroman
\item
A \textbf{bilinear form} is a $\preceq$-bilinear form where $=$
replaces $\preceq$, and as in the classical theory is obtained by an
$n\times n$ matrix $B$; the form is symmetric iff $B$ is a symmetric
matrix. In particular, one could take the 0 form.
\item   Suppose $A$ is zero sum free, in the sense that $a+a' =0 $ implies $a=a'=0.$ Fix a base
$\{ x^1, \dots, x^n\}$ of $V$ and define $\ind(v)$ for a vector $v =
\sum a _i x^i$ to be the number of nonzero coefficients. Given a
  ${\preceq}$-bilinear form $\Bb:
V\times V\to {\mathcal A}$, define
$$\Bb'(v,w) = {\ind(v)\ind(w)}\Bb(v,w)^{\circ}.$$ Since $\ind
(v_0+v_1) \ge \max \{\ind(v_0) ,\ind(v_1)\}$ we see that

 \begin{equation} \label{Miy24} \begin{aligned}
\Bb'(a_0 v_0 + a_1 v_1, w)& = {\ind(v_0+v_1)\ind(w)} \Bb(a_0 v_0 +
a_1 v_1, w)^{\circ} \\ & \succeq {\ind(v_0+v_1) \ind(w)}(a_0
\Bb(v_0,w)^{\circ} + a_1 \Bb(v_1, w)^{\circ})\\ & \succeq {\ind(v_0)
\ind(w)}a_0 \Bb(v_0,w)^{\circ} + \ind(v_1)\ind(w) a_1 \Bb(v_1, w)^{\circ}\\
& \succeq a_0 \Bb'(v_0,w) + a_1 \Bb'(v_1, w).\end{aligned}
\end{equation}

\item The sum of two ${\preceq}$-bilinear forms is a ${\preceq}$-bilinear
form, so (ii) gives us a considerable range of examples.
\end{enumerate}
\end{example}

The reason we use $\preceq$ is to prepare for
Construction~\ref{stdG}, Construction~\ref{stdG0}, and
Construction~\ref{stdGN} below; its use is consistent with the next
definition.

\begin{defn}
 A ${\preceq}$-\textbf{quadratic
form} on a  systemic ${\mathcal A_0}$-module  $(V, \tT_{V},
(-),\preceq)$ is a function $q: V\to {\mathcal A_0}$ with
\begin{equation}\label{eq:I.1.11} q(av)\succeq a^2q(v)\end{equation}
 for any $a\in {\mathcal A},$ $v\in V,$ together with a
symmetric ${\preceq}$-bilinear form $\Bb: V\times V\to {\mathcal
A}$ (not necessarily uniquely determined by $q$) such that for any
$ v,w\in V,$
\begin{equation}\label{eq:I.1.27} 2 q(v) = \Bb(v,v),\qquad q(v+w)\succeq q(v)+q(w)+\Bb(v,w).\end{equation}
  We call $(q ,\Bb)$ a
${\preceq}$-\textbf{quadratic pair} over ${\mathcal A}$.

A \textbf{quadratic pair}  is a ${\preceq}$-quadratic pair $(q
,\Bb)$ for which $\Bb$ is a bilinear form and $q(v+w)=
q(v)+q(w)+\Bb(v,w)$ for all $v,w \in V.$
\end{defn}

\claim{\bf Abstract definition of Clifford
${\preceq}$-semialgebra.}

Just as Clifford algebras are essential to the theory of quadratic
forms, we need a systemic version for this paper and for further
research in quadratic forms. As we shall see, this is a more
delicate issue than exterior semialgebras. We start with the desired
defining properties.

\begin{defn}\label{Cliff1} A \textbf{Clifford ${\preceq}$-semialgebra} $\mathcal C(q,B,V)$ of  a ${\preceq}$-quadratic pair $(q,B)$ over ${\mathcal A}$,
 is a semialgebra~$\mathcal C$ generated by  ${\mathcal A}$ and $V$,
together with  a  product $\mathcal C \times \mathcal C \to \mathcal
C$ satisfying \be \label{sq} v^2 \succeq q(v) \qquad  \text{for}
\qquad  v\in V,\ee
      \be \label{wn} v_1v_2+
 v_2v_1 \succeq \Bb(v_1,v_2), \quad  \forall v_i\in V.
\ee
\end{defn}

Lacking negation, it could be unlikely to achieve equality in
\eqref{wn}. Nonetheless we have:
\begin{constr}\label{stdG} The \textbf{standard  Clifford
semialgebra} $\widetilde{\mathcal C}(q,B,V)$ is defined by imposing
the relation $\preceq$ on $T(V)$ via Lemma~\ref{makesur}, taking
$$S_0 = \{\left(\Bb(v_1,v_2),v_1\otimes v_2+
 v_2\otimes v_1  \right),\ \left( q(v), v\otimes v \right):\, v, v_i \in V\},$$ and $S$ the $\tT$-submodule
of $T(V)\times T(V)$ generated by $S_0$.\end{constr} Thus a standard
Clifford semialgebra becomes an exterior semialgebra of type~1 when
$q = B= 0.$  But this does not yet have a negation map when $q \ne
0$. When $\Acal$ already possesses a negation map, we can define a
negation map by defining $(-)v = ((-)\one)v$ and ($-)(v_1\otimes
v_2\otimes\dots \otimes v_k) .$

 Accordingly, we equip $\Acal$ with a negation map. This can just be
 the identity, as in tropical algebra, or else we use
  to the ``symmetrization process,'' which we recall from
 \cite[Definition~3.6]{Row16}:

  \begin{definition}\label{sym00}
Define $\widehat {\mathcal A} = \mathcal A \times \mathcal A$
with componentwise addition, with multiplication
  given by the \textbf{twist action}
$$(a_0,a_1)(b_0,b_1) = (a_0 b_0 + a_1 b_1, a_0 b_1 + a_1 b_0), \quad a_i, , b_i \in \mathcal A_0.$$
The negation map on $\widehat {\mathcal A}$ is the \textbf{switch
map} $ (-)_{\operatorname{sw}}$ given by
 \be\label{C10}(-)_{\operatorname{sw}}(b_0,b_1) = (b_1,b_0).\ee The surpassing
relation $\preceq_{\circ}$
 is given by:

\begin{equation} (b_0,b_1) \preceq_{\circ}  (b_0',b_1')  \quad
\text{iff} \quad  b_i'   = b_i + c  \  \text{for some } c \in
{\mathcal A}, \ i = 0,1.\label{eq:shneg}
\end{equation}
(The same $c$ is used for both components.)
 \end{definition}

So we first symmetrize $\Acal$ and $V$ to $\widehat {\mathcal A}$
and $\widehat V$, and then form $\mathcal C = \widetilde{\mathcal
C}(q,B,\widehat V) $ and the Clifford system $({\mathcal C}
,\tT_{\mathcal C}, (-), \preceq_\circ)$, with $(-)$ as in
\eqref{C10}.

\begin{lem}  $\mathcal C$ as defined in the paragraph above is a Clifford semialgebra, and  $({\mathcal C}
,\tT_{\mathcal C}, (-), \preceq_\circ)$ is a  system.\end{lem}
\proof It follows from the fact that \eqref{C10} defines a negation
map. \qed

 We can construct a Clifford
semialgebra a little more efficiently using ideas of
\S\ref{combcon}.

\begin{constr}\label{stdG0} The \textbf{standard   Clifford
semialgebra with negation map} $\mathcal C(q,B,V,(-))$ is defined as
follows:

Define the
  tensor module~$T_{\ge 0}(\V )$ of a free
${\mathcal A}$-module $\V $, as the submodule of $T(\V )$ spanned
by $\{ x^{i_1} \otimes \dots \otimes x^{i_k}: k \in \Net   , \ i_1
\le i_2 \le \dots \le i_k \}$ and take $T = \Acal \oplus T_{\ge
0}(\V ) \oplus (-)T_{\ge 0}(\V )$ (where $(-)T_{\ge 0}(\V )$ is
another copy of $T_{\ge 0}(\V )$) under multiplication
\be\label{Cliff00}  x^j x^i = \Bb(x^i,x^j) + (-) x^i \otimes x^j, \
\forall j>i,\ee and inductively
 \be\label{Cliff07}(x^{i_1} \otimes \dots \otimes x^{i_k})\otimes x^j =
 \begin{cases}
 x^{i_1} \otimes \dots \otimes x^{i_k}\otimes x^j , \quad k\le j  \\
\Bb(x^k,x^j) x^{i_1} \otimes \dots \otimes x^{i_{k-1}} + ((-)
x^{i_1}
\otimes \dots \otimes x^{i_{k-1}}  \otimes x^j \otimes x^{i_k} ),\\
\quad k>j ,
 \end{cases}\ee
 where $(-)x^i = x^i,$
extended distributively.
\end{constr}

\begin{theorem} $\mathcal C(q,B,V,(-))$ of Construction~\ref{stdG0} is a Clifford semialgebra with negation map.
 Define \be\begin{aligned} \preceq \ : = \ \preceq_\circ   \cup  \bigcup _{k
  \in\Net^+} \{x^{i_1} \otimes \dots \otimes x^{i_j}\otimes x^{i_j}\otimes
 \dots \otimes x^{i_k},\ \ q( x^{i_j}) x^{i_1} \otimes  \dots \otimes x^{i_k}\},\end{aligned}\ee  $i$ running over $\Net^+,$
  where $ x^{i_j}$ is deleted in the tensors on the right side. (In other words we declare that $ x^{i_j}\otimes  x^{i_j}
  \preceq q( x^{i_j})$.) 

  Taking
$\tT_{\mathcal C(q,B,V,(-))}$ to be the simple tensors and their
negations, $$(\mathcal C(q,B,V,(-)),\tT_{\mathcal
C(q,B,V,(-))},(-),\preceq)$$ is a
  system.
   There is an injection $\Phi: \widetilde{\mathcal
C}(q,B,V)\to {\mathcal C}(q,B,V,(-))$ given by $$ x^{i_1} \otimes
\dots \otimes x^{i_k}\mapsto (-)^\pi  x^{\pi(i_1)} \otimes \dots
\otimes x^{\pi(i_k)},$$ where $\pi$  is the permutation rearranging
$i_1,\dots, i_k$ in ascending order.
\end{theorem}\proof We start by verifying the conditions of
Definition~\ref{Cliff1}.

 $x^i \otimes x^j + x^j \otimes x^i = \Bb(x^i,x^j) + (x^i
\otimes x^j)^\circ \succeq \Bb(x^i,x^j)$ for $i<j$, so we take sums.

To prove associativity of multiplication, we need only check tensors
of basis elements, and thus, by induction on length, we appeal to
\eqref{Cliff07}. \qed

 Unfortunately this construction is not finitely spanned over $\Acal$, since
 the tensor powers of $x^i$ are independent. We can reduce $x^i \otimes x^i$ to $q(x^i)$, at the cost of losing associativity.
Accordingly, we  weaken associativity.
\begin{defn}  A nonassociative semialgebra $\mathcal C$ with a surpassing relation $(\preceq)$ is $\preceq$-\textbf{associative} if $a_1(a_2a_3) +
(a_1 a_2)a_3 \succeq \zero,$ and $\mathcal C$ is called a
$\preceq$-\textbf{semialgebra}.
\end{defn}

\begin{constr}\label{stdGN} The \textbf{reduced   Clifford
$\preceq$-semialgebra with negation map} $\mathcal C_\preceq
(q,B,V,(-))$ is defined as follows:

Define the
  tensor module~$T_{>}(\V )$ of a free
${\mathcal A}$-module $\V $, as the submodule of $T(\V )$ spanned by
$\{ x^{i_1} \otimes \dots \otimes x^{i_k}: k \in \Net , \   : i_1 <
i_2 < \dots < i_k  \}$, and take $$T := \Acal \oplus T_{>}(\V )
\oplus (-)T_{>}(\V )$$ (where $(-)T_{>}(\V )$ is another copy of
$T_{>}(\V )$) under multiplication \be\label{Cliff04} x^i \otimes
x^i = q(x^i),\qquad x^j x^i = \Bb(x^i,x^j) + ((-) x^i \otimes x^j), \
\forall j>i,\ee and inductively
 \be\label{Cliff7}(x^{i_1} \otimes \dots \otimes x^{i_k})\otimes x^j =
 \begin{cases}
 x^{i_1} \otimes \dots \otimes x^{i_k}\otimes x^j , \quad k<j \\
q(x^j) x^{i_1} \otimes \dots \otimes x^{i_{k-1}}, \quad k=j \\
x^{i_1} \otimes \dots \otimes x^{i_{k-1}}\Bb(x^k,x^j) + ((-) x^{i_1}
\otimes \dots \otimes x^{i_{k-1}}  \otimes x^j \otimes x^{i_k} ),\\
\quad k>j ,
 \end{cases}\ee
 where $(-)(-)x^i = x^i,$
extended distributively.
\end{constr}

\begin{rem} Associativity in the reduced   Clifford
$\preceq$-semialgebra with negation map fails in general since
$$ x^2 \otimes (x^2 \otimes x^1) = x^2 \otimes (\Bb (x^1 ,x^2) (-)x^1 \otimes
x^2) = q(x^2) x^1 + \Bb (x^1 ,x^2) (x^2)^\circ$$ whereas $(x^2
\otimes x^2) \otimes x^1 =  q(x^2) x^1.$ Likewise
$$ (x^2 \otimes x^1 )\otimes x^1 = \Bb (x^1 ,x^2) (-)x^1 \otimes
x^2) \otimes x^1 = q(x^1) x^2 + (\Bb (x^1 ,x^2) (x^2)^\circ$$
whereas $x^2 \otimes (x^1 \otimes x^1) = q(x^1)x^2.$
\end{rem}

Fortunately in our application of representing the exterior algebra,
we do have associativity since $\Bb$ is zero where this difficulty
would arise.

\begin{theorem}  $\mathcal C_\preceq(q,B,V,(-))$ is a Clifford $\preceq$-semialgebra.
 $$(\mathcal C_\preceq(q,B,V,(-)),\tT_{\mathcal
C(q,B,V,(-))},(-),\preceq)$$ is a
  system,  taking
$\tT_{\mathcal C(q,B,V,(-))}$ to be the simple tensors and their
negations,    \end{theorem}\proof We start by verifying the
conditions of Definition~\ref{Cliff1}.
  $$x^i \otimes x^j + x^j \otimes x^i = \Bb(v,w) + (x^i
\otimes x^j)^\circ \succeq \Bb(v,w)$$ for $i<j$, so we take sums.

To prove $\preceq$-associativity of multiplication, we need only
check tensors of basis elements, and thus, by induction on length, we
appeal to \eqref{Cliff7}. \qed

 Since the indices increase, there must be at most
$n$ of them in a product of basis elements. Thus the dimension is
$2^{n +1} $ rather than  $2^n$ in the classical case, because we
need to count negatives as well. Note that
 $\widetilde{\mathcal C}(q,B,V)\to {\mathcal C}_\preceq(q,B,V,(-))$ is no longer an injection
 since $x^i \otimes x^i $ and $ q(x^i )$ have the same image.

\begin{rem}
In summary, the standard Clifford algebra is associative but lacks
the important classical property of finite generation. The reduced
standard Clifford algebra seems more in line with the spirit of this
paper, replacing equality by $\preceq$ in the key property of
associativity.

When $\mathcal{A}$ is taken to be a field, the formal negation map
coincides with the negation in the field, and thus both the standard
and the reduced Clifford algebras boil down to the classical
Clifford algebra of a quadratic form over a field.
\end{rem}

\begin{example}[The standard  reduced Clifford $\preceq$-semialgebra with $(-)$ for
$n=2$]\label{std00} We consider the base
$\{(\pm)1,\ x:= x^1, \ (-)x,\ y: = x^2,\ (-)y\}$ over $\Acal,$ and
given a quadratic pair $(q,\Bb)$,    put $\gamma = \Bb (x,y).$ Then
$x\cdot x = q(x),$ $y \cdot y = q(y),$ and $y\cdot x = (-)x\cdot y +
\gamma.$
\end{example}

%
%
%
%
%
%
%

We turn to involutions.
\begin{defn}  A $\preceq$-\textbf{involution} $(\sigma)$ of a  $\preceq$-
system  $\mathcal A $ is a $\preceq$-antihomomorphism of
order $\le 2,$ in the sense that  $(a^\sigma)^ \sigma \succeq a$;
\be \label{inv0} (a_0 a_1)^\sigma \preceq a_1^\sigma
a_0^\sigma,\quad \forall a,a_i \in \tT.\ee
\end{defn}

\begin{lem}   The standard    Clifford
$\preceq$-semialgebra with negation map $\mathcal C(q,B,V,(-))$ has
an involution $\sigma$ given by the identity on $\Acal$ and
$v^\sigma = (-)v$ and
$$(v_1 \otimes v_2 \otimes\dots \otimes v_k)^\sigma = (-)^k v_k
\otimes \dots \otimes v_2 \otimes v_1  .$$\end{lem}

\begin{proof} It is enough to check \eqref{inv0}  for tensors of
basis elements.  For $j>i_k,$ \be\begin{aligned}((x^{i_1}& \otimes
\dots \otimes x^{i_k})\otimes x^j)^\sg = (-)^k x^{i_k} \otimes \dots
\otimes x^{i_2} \otimes x^{i_1} \\& = (-)^k x^{i_k} \otimes
((-)^{k-1}x^{i_{k-1}}\dots \otimes x^{i_2} \otimes x^{i_1}) =
{x^j}^\sg \otimes (x^{i_1} \otimes \dots \otimes
x^{i_k})^\sg.\end{aligned}
\end{equation} For $j=k,$
$$((x^{i_1} \otimes \dots \otimes
x^{i_k})\otimes x^{i_k})^\sg = q(x^{i_k}) (x^{i_1} \otimes \dots
\otimes x^{i_{k-1}})^\sg = {x^k}^\sg \otimes (x^{i_1} \otimes \dots
\otimes x^{i_k})^\sg.$$
 For $j<k,$ \be\begin{aligned}((x^{i_1}& \otimes \dots
\otimes x^{i_k})\otimes x^j)^\sg = ((x^{i_1} \otimes \dots \otimes
x^{i_{k-1}})\otimes (\Bb (x^j, x^{i_k})(-)(x^j \otimes x^{i_k})^\sg
\\& =
(-)^{i_{k-1}} \Bb (x^j, x^{i_k})x^{i_{k-1}} \otimes \dots \otimes
x^2 \otimes x^1  + (-)^{k+1} x^{i_k} \otimes  x^j \otimes \dots
\otimes x^2 \otimes x^1
\\& =
(-)^{k-1} \Bb (x^j, x^{i_k})^\circ x^{i_{k-1}} \otimes \dots \otimes
x^{i_2} \otimes x^{i_1}  + (-)^{k+1} x^{i_k} \otimes  x^j \otimes
\dots \otimes x^{i_2} \otimes x^{i_1} \\& =  (x^j)^\sg (x^{i_1}
\otimes \dots \otimes x^{i_k})^\sg.\end{aligned}
\end{equation}
\end{proof}

\begin{example} The involution for Example~\ref{std00} is given by
$x^\sg= (-)x,$  $y^\sg =(-)y,$ $(xy)^\sg = yx = (-)x y + \gamma,$
$(yx)^\sg = xy.$ Thus  $$(xy)^\sg  = ((-)x y + \gamma)^\sg = x y\,
(-) \gamma+ \gamma \succeq xy$$ and $$((xy)y)^\sg = q(y) x^\sg =
x^\sg q(y) \preceq xy^2 +(\gamma y)^\circ = y(xy) -\gamma y) = y^\sg
(xy)^\sg
$$
$$((xy)x)^\sg = ((-)x^2 y +\gamma x)^\sg = q(x)y  +\gamma x
=  x(xy) -\gamma x = x^\sg (xy)^\sg
$$
\end{example}

\section{The dual space and the Clifford semialgebra}\label{sec:sec6}

We  bring in our major use of the Clifford semialgebra, removing
$\otimes$ from the notation. For simplicity, we continue to deal
with the free $\Acal$-module $\V=\Acal[x]$ of infinite countable
rank with base $ \bfx:=(x^0,x^1,x^2,\ldots)$, the power of the
indeterminate $x$.
\begin{defn}
Suppose we are given  an $(\Acal,\preceq)$-bilinear form $B$ and a
module $ \V = ( \V, \tT_{\V }, (-)))$ over $\Acal.$ Define $
\d^j:V\to \Acal$  by $\d^j(v):=B(u,x^j)$ for all $v\in V$. In other
words, the restricted dual $V^*$ will be the $\Acal$-span of the
$B$-dual basis $\d^j$ such that $\d^j(x^i)=B^{ij}=B(x^i,x^j)$.
Symbolically we could also denote $V^*$ by~$\Acal[\d]$.  If
$B(x^i,x^j)= \delta_{ij}$ then $\d^j$ is the element of the usual
dual base.  Write $\tT _{\V  }^*$  for   $\{ f|_{\tT_{\V }} : f \in
\V  ^*$ with $f(\tT_{\V }) \subseteq \tT_{\mathcal A}\}.$\end{defn}

Define the set $\Wcal(\bfx,\partial)=\Wcal(\bfx)\cup\Wcal(\partial)$ of
all the words in the union of the two infinite alphabets
$\Wcal(\bfx)=\{x^0, x^1, x^2,\}$ and
$\Wcal(\partial)=\{\partial^0,\partial^1,\partial^2,\ldots, \}$.
Recall that if $v_1,v_2\in V$ and $w\in \tbw^nV$, then
$$
(-)v_1\w (v_2\w w)=v_2\w v_1\w w,
$$
 which defines a negation map on $\tilde V$ compatible with that on $\tbw^{\geq 2}V$.

Denote by
$\Ccal= \Ccal(V)$ the set of formal finite linear combinations
of words of $\Wcal(\bfx,(-)\bfx, \d, (-)\partial)$ with $1_\Acal$ adjoined, with
respect to the juxtaposition product extended distributively, and
with relations given  (for $i\ge j$) by
\be
x^ix^j  = (-)x^jx^i, \qquad \partial^i\partial^j =
(-)\partial^j\partial^i, \qquad \partial^j x^i= \Bb(x^i,x^j)
 (-) x^i\partial^j.\label{eq:cacr}
\ee

Define $\Ccal:=\Ccal(V)$ as being the $\Acal$-algebra generated by
$1_{\Acal}\cup \Ccal_+\cup\Ccal_-$, where $\Ccal_+$ and $\Ccal_-$
are graded:
$$
\Ccal_+=C_1\oplus_{h\geq 1}\Acal\cdot \underbrace{x^{i_1}\cdots x^{i_h}}_{words\,\, of\,\, length\,\, h}
$$
where
$$
\Ccal_1= \bigoplus_{j\geq 0}\Acal x^j\oplus \bigoplus_{j\geq
0}\Acal((-)x^j)
$$
and

$$
\Ccal_-=C_{-1}\oplus_{h\geq 1}\Acal\cdot \underbrace{\d^{i_1}\cdots \d^{i_h}}_{words\,\, of\,\, length\,\, h}
$$
where
$$
\Ccal_{-1}= \bigoplus_{j\geq 0}\Acal\cdot \d^j\oplus
\bigoplus_{j\geq 0}\Acal((-)\d^j),
$$
 defining
$$
(-)\d^j(x^hx^k)=\d^j((-)x^hx^k)=\d^j(x^kx^h).
$$
If $u(x) = \sum_i a_ix^i \in V=\Acal[x]$ we define its $B$-dual over
$V^*$ as $u(\d): = \sum_i a_i \d^i$, i.e.
$u(\d)(v(x))=B(u(x),v(x))$. We make $\Ccal$ into a semialgebra over
$\Acal$ with respect to the juxtaposition product, by imposing the
following commutation rules on the elements of degree $\pm 1$:
\begin{eqnarray}
u(x)v(x)&+&v(x)u(x)\succeq 0,\cr\cr u(\d)v(\d)&+&v(\d)u(\d)\succeq
0,\cr\cr u(\d)\cdots v(x)&\succeq&  u(\d)(v(x)) (-) v(x)u(\d) ,\label{eq:comm3}
\end{eqnarray}
which basically descend from the fact that we are  requiring the elements of $V^*$ acting  on
$\Ccal$ as skew-derivations, namely \be u(\d)(v(x)w(x))=(u(\d)v(x))w(x)(-)v(x)u(\d)w(x)\label{eq:der}. \ee To
be more precise, equality (\ref{eq:der}) should be understood in the
sense of operators, i.e. for all $f\in \Ccal_{1}$:
$$
(u(\d)v(x)w(x)(-)v(x)u(\d)(w(x)))f=u(\d)(v(x))f+u(\d)(w(x))fv(x).
$$
Since each element of $\Ccal$ is a finite linear combination of
words of elements of degree $1$, the above suffices for the
definition of
the commutation relations on $\Ccal$.
\begin{example}
 This example serves to motivate commutations (\ref{eq:comm3}). Let $\tbw V$ be the exterior semi algebra. Any element of degree $\geq 2$ is a sum of elements of the form $w:=w_1\w w_2$ with $w_1,w_2\in  V$  (possibly equal to $1$).  Now $u(\d)$ acts on $w:=w_1\w w_2$ as
\begin{eqnarray*}
 u(\d)(w_1\w w_2)&=&u(\d)(w_1)\w w_2(-)w_1\w u(\d)(w_2)
\end{eqnarray*}
which means that
$$
  u(\d)(w_1\w w_2)\w w_3=u(\d)(w_1)(w_2\w w_3)+  u(\d)(w_2) w_1\w w_3.
$$
This  motivates our requirement (\ref{eq:der}). Furthermore a simple computation give:
$$
(u(\d)v(x)+v(x)u(\d))w\succeq B(u(x),v(x))w
$$
i.e.
$$
u(\d)v(x)+v(x)u(\d)\succeq B(u(x),v(x))
$$
or
$$
u(\d)v(x)\succeq B(u(x),v(x))(-) v(x)u(\d)
$$
whence the  commutation relation in the Clifford semialgebra
$$
u(\d)v(x)=u(\d)(-)v(x)u(\d)
$$
\end{example}

\claim{} For notational simplicity write $u^*:=u(\d)\in V^*$ and $v:=v(x)\in
V$. We also denote by ``$\cdot$'' the product of $\Ccal$ to avoid
potential confusions.  The product $u^*v$ in $\Ccal$ defines a linear
map $\Ccal\sra \Ccal$ given by \be (u^*\cdot v)w=u^*(v)(w)=u^*(vw).
\ee Then
\begin{eqnarray*}
(u^*\cdot v)\cdot w&=&u^*(v)\cdot w(-)u^*(w)v
\cr\cr
\end{eqnarray*}
which means that
$$
[(u^*\cdot v)\cdot w]z=u^*(v)w\cdot z+u^*(w)z\cdot v.
$$
Notice that by construction the product in $\Ccal$ is associative
because we are composing operators, and the  composition of
operators is associative.


 \begin{rem}{\bf {(Negation on $\Acal$)}
 .}\label{Cliffext}   We endow $\Acal$ with a
negation map, by setting $((-)1_\Acal) u=(-u)$ in the
 sense of operators, namely $((-)u)v=vu$.
 Write $\widetilde{\Acal} = \Acal \oplus ((-)\Acal.$
$\widetilde{\Acal}$ has a negation map $(-)$ given by the switch.
Likewise, take a copy $(-)V$ of $V$, and view $\widetilde{V}:= V
\oplus((-)V)$ as an $\widetilde{\Acal}$-module via
$$(a_1,a_2)(v_1,v_2) = (a_1v_1+a_2v_2, a_2v_1+a_1v_2),$$
and $(-)(v_1,v_2):= (v_2,v_1).$
\end{rem}

 We take the
dual $\widetilde{V}^*$ of $\widetilde{V}$ which can be thought of as
pairs $w^*:=(w_1^*,w_2^*),$ where for $v=(v_1,v_2)\in \widetilde{V},$
$$w^*(v) = (w_1^*(v_1)+w_2^*(v_2),w_1^*(v_2)+w_2^*(v_1)).$$
Then $((-)w)^*(v) = (w_1^*(v_2)+w_2^*(v_1),w_1^*(v_1)+w_2^*(v_2)).$

 Notice that
$$
(-)u\cdot v=v\cdot u=(-)(-)v\cdot u=u\cdot (-)v
$$
so that the relations $(-)1_\Acal(u\cdot v)=(-)u\cdot v=u\cdot(-)v$
hold. We extend the bilinear form $B$ to $V \oplus (-)V,$ by putting
$$B( (x^i,(-)x^{i'}) ,(x^j,x^{j'})) = (B(x^i,x^j)+B(x^{i'},x^{j'})(-)(B(x^i,x^{j'})+B(x^{i'},x^j))$$ Now we define
$(-)B(x^i,x^j) :=B((-)x^i,x^j),$ and our negation map also induces a
negation map on $V^*$ in $\Ccal$, given by $(-)(x^i)^*(x^j) = (-)
\Bb(x^i,x^j)=B((-)x^i,x^j).$

\begin{definition} An element of $\Ccal$ is in \textbf{normal form} if it is a finite linear combination of words $w_1w_2$ such that $w_1\in \Wcal(\bfx)$ and $w_2\in\Wcal(\bm\partial)$.
\end{definition}
\begin{example} Here are a couple of examples of the basic yoga to put products in normal form using the appropriate commutation relations.
\begin{enumerate}
\item To put $u_1^*u_2^*v$ in normal form we exploit the fact that the elements of $V^*$ act as derivations on $\tbw V$. One has
\begin{eqnarray*}
u_1^*\cdot u_2^*\cdot v&=&u_1^*\cdot(u_2^*\cdot v)=u_1^*\big(B(u_2,v)+((-)v)\cdot u_2^*\big)\cr
&=&B(u_2,v)u_1^*+B(u_1, (-)v)u_2^*+v\cdot u_1^*\cdot u_2^*
\end{eqnarray*}
where we agree that $B(u, (-v))=(-)B(u,v)$, acting on the product $w-1\cdot w_2$ as
$(-)B(u,v)w_1w_2=B(u,v)w_2\cdot w_1$

\item  As a second example we propose the usual expression one may happen to deal with, such as the monomial $
x^i\partial^jx^k$. To put it in normal form:
$$
x^i\partial^jx^k=x^i(B_{jk}\, (-) \,x^k\partial^j)=B_{jk}x^i(-) \,x^k\partial^j
$$
where $B_{ij}:=B(x^i,x^j)$.
\end{enumerate}
\end{example}

\begin{proposition}
All elements  of $\Ccal$ can be put into normal form.
\end{proposition}
\proof A matter of a routine  induction.\qed

We do not claim that the normal form is unique, but we do have
uniqueness up to $\circ$: If $\sum a_i w_{1,i}w_{2,i} =  a_i'
w_{1,i}w_{2,i}$ for $a_i, a_i' \in \Acal$ then each $a_i (-) a_i'
\in \Acal^\circ.$

 Next, we apply
some material from \cite[\S 9]{JuR1} to the Clifford semialgebra.

 \begin{rem}\label{dua} $(V_{\Bb}^*,\tT_V^*,\preceq)$ is a systemic module, where we define
 $v_1^* \preceq v_2^*$ if $v_1^*(w) \preceq v_2^*(w)$ for every $w\in V.$ \end{rem}

\claim{\bf The action  on $\tbw\V $ (and on $\overline{\bw\V
}$)}\label{sec:sec2}

 We have the following:


%
\begin{definition}\label{def:conr}  Define   a left  action $\, \lrcorner \,$ (possibly depending on the choice of the $\preceq$ bilinear form $B$) of
$\V _n^*$ on $\bw\V _n$ by:
$$
\d\, \lrcorner \, v=\d(v),\qquad\qquad \forall (\d,v)\in\V
_n^*\times\V _n,
$$
and for $(u_1,\ldots,u_k)\in \V _n^k$, \be \d\, \lrcorner \,
(u_1\w\cdots\w u_k)=\begin{vmatrix}\d(u_1)&\d(u_2)&\ldots&\d(u_k)\cr
u_1&u_2&\ldots&u_k\end{vmatrix}\in \widetilde{\bw}_\Acal^{k-1}\V
_n,\label{eq:contdiag} \ee where the expression in the RHS of
\eqref{eq:contdiag} means the linear combination
 $$\sum_j (-)^j \d(u_j) u_1\w\cdots\w u_{j-1} \w u_{j+1}\w\cdots\w u_k \in \widetilde{\bw}_\Acal^{k-1}\V _n.$$
\end{definition}

The cases $k=2$ and $3$ merit  explicit descriptions.
\begin{example}
 Let $u\w v\in
\bw^2\V _n$ and $\d\in\V _n^*$. Then
$$
\d\, \lrcorner \, (u\w v)=\begin{vmatrix}\d(u)&\d(v)\cr u&
v\end{vmatrix}=\d(u)v(-)\d(v)u\in \tbw^1_\Acal\V _n,
$$
where
$$
(\d(u)v(-)\d(v)u)(w\w w_1)=\d(u)\cdot v\w w\w w_1+\d(v)\cdot w\w u\w
w_1\qquad (w,w_1)\in \V _n\times \bw\V _n.
$$
Similarly:
$$
\d\, \lrcorner \, (u\w v\w w)=\begin{vmatrix}\d(u)&\d(v)&\d(w)\cr
u&v&w\end{vmatrix}=\d(u)v\w w(-)\d(v)w\w u+\d(w)u\w v.
$$
\end{example}

\begin{lemma} For all $\d\in \V _n^*$, and $v_1,v_2,\ldots, v_k$ ($k\geq 2$),
denote
$$
\begin{vmatrix}\d(\bfv)\cr \bfv\end{vmatrix}:=\begin{vmatrix}\d(v_1)&\cdots&\d(v_k)\cr v_1&\ldots&v_k\end{vmatrix}
$$
Then
$$
\d \, \lrcorner \, (u\w v_1\w\cdots\w
v_k)=\begin{vmatrix}\d(u)&\d(v_1)&\cdots&\d(v_k)\cr
u&v_1&\cdots&v_k\end{vmatrix}=\d(u)v_1\w\cdots\w v_k\ (-)\  u\w
\begin{vmatrix}\d(\bfv)\cr \bfv\end{vmatrix}.
$$
\end{lemma}
\proof By direct substitution. \qed

 The notation of \eqref{eq:contdiag} is useful in simplifying the
combinatorics.

\begin{proposition} For  $\d \in \V _n^*$, the map $\d\, \lrcorner \,: \tbw\V \sra \tbw\V $ is a $V$-skew derivation, i.e., for $u\in \tbw^1V$ and
$v\in\bw^kV$, \be\label{skew} \d\, \lrcorner \,(u\w v)=\d(u)v\ (-)\
u\w (\d\, \lrcorner \, v). \ee
\end{proposition}

\proof The map $\d$ is clearly $\Acal$-linear. On $u\w v\in \bw^2\V
_n$ one has:
$$
\d \, \lrcorner \, (u\w v)=\begin{vmatrix}\d(u)&\d(v)\cr
u&v\end{vmatrix}=\d(u)v\, (-)\, \d(v)u.
$$
Assume that the assertion holds for $k-1\geq 1$ and let $v\in
\bw^jV$ with $j\geq 2$. Then
$$
\d \lrcorner ( u\w v)=\d(u)v (-)\ u\w\!
\begin{vmatrix}\d(v)\cr v\end{vmatrix},
$$
which proves the assertion for $k$.\qed

Another way of interpreting \eqref{skew} is $ \d  \,u =\d(u)  (-)\,
u  \d,$ which implies that for all $j\geq 0$, $$\d^j  \,u + u\d^j
\succeq \d^j(u) = \Bb(u,x^j).$$ This is our desired connection with
Clifford algebras, which we now formulate.

%

\claim{\bf The Clifford Algebra action on $\bw V$.}
The words of the
form $ x^{i_1}\cdots x^{i_k} $ act on $\bw\V $ by wedging:
$$
x^{i_1}\cdots x^{i_k}(u)=x^{i_1}\w\cdots\w x^{i_k}\w u,\qquad \forall
u\in \bw\V.
$$
The words involving only the $\partial_j$ act by contraction:
$$
\d^{j_1}\cdots\d^{j_l}u=\d^{j_1}\, \lrcorner \,(\cdots\, \lrcorner
\, (\d^{j_k}\, \lrcorner \, u))\cdots)
$$

 Define
$<,>:\mathcal C_{\pm}\to \Acal$  by
$$
<u_1\oplus v_1^*,u_2\oplus v_2^*>\, =\, B(u_1,v_2)+B(u_2,v_1)
$$
This is a nondegenerate inner product on $V\oplus V^*$.

\begin{proposition}
The inner product possesses an orthogonal base.
\end{proposition}
\proof Let us consider $x^i:= (x^i\oplus 0)$, $(x^j)^*:=(0\oplus
\partial_j)$, so that $x^i\oplus (x^j)^*=x^i+(x^j)^*$ and $\partial_j (e_i)=\Bb
(e_i, e_j)$. Then $(x^0,\ldots, x^{n-1}, {x^0}^*,\ldots,
{x^{n-1}}^*)$ is an orthogonal base.


\claim{\bf The Clifford  representation of an exterior semialgebra.
}\label{sec:sec4}

\begin{theorem}
The exterior semialgebra is an irreducible representation of the
Clifford semialgebra $\Ccal$.
\end{theorem}
\proof First of all we show that for any choice of a
$\preceq$-bilinear form $B$, the exterior semialgebra represents
$\Ccal$. It is obvious that each word of of $\Ccal$
defines an $\Acal$-endomorphism of $\tbw V$. We have to check the
commutation relations. It is obvious that the relation $u\cdot
v+v\cdot u$ is mapped to the endomorphism $u\w v +v\w v$ of $\tbw
V$, which is a quasi-zero in $\End_\Acal(\tbw)$. The same holds with
$u^*v^*+v^*u^*$ acting on $w$ as $u^*\lrcorner(v^*\lrcorner
w)+w^*\lrcorner (v^*\lrcorner w)\geq 0$. It remains to check the
action of $u^*v+v\cdot u^*$.  Let us consider $w:=w_1\w w_2$ where
$w_1\in V$ and $w_2\in\tbw V$. Then
$$
u^*(v\w w_1\w w_2)+v\w (u^*( w_1\w w_2))=u^*(v)w_1\w w_2+v\w u^*(w_1)w_2 (-)
$$

To prove the irreducibility, suppose on the contrary that the
exterior semialgebra has a proper invariant sub-module $W$ under the
action of $\Ccal$. Let $U$ be the submodule of $W$ of elements of
minimal degree, say $r$, whose typical elements are finite
$\Acal$-linear combinations of $u_1\w\cdots\w u_r$. But for any
$v\in V$, $vU$ is a submodule of elements of degree $r+1$, and thus
$W$ cannot be invariant.\qed

\begin{proposition}\label{action}
There is an $\Acal$-semialgebra homomorphism $$\gl(\tbw^1\V
):=\tbw^1_\Acal\V \otimes_\Acal\tbw^1\V _\Acal^*\sra
\End_\Acal(\bw^1\V ).$$
\end{proposition}
\proof A general element of $\V \otimes_\Acal\V ^*$ is  of the form
$$
\displaystyle{\sum_{0\leq i,j\leq n-1}}x^i\otimes
(a_{ij}\d_j+c_{ij}((-)\d_j))+ \displaystyle{\sum_{0\leq
i,j\leq n-1}}(-)x^i\otimes (a_{ij}'j+c_{ij}'((-)\d_j)),
$$
which gives the endomorphism
$$
u\mapsto \displaystyle{\sum_{0\leq i,j\leq
n-1}}x^i(a_{ij}\d_j(u)+c_{ij}((-)\d_j(u)))+
\displaystyle{\sum_{0\leq i,j\leq n-1}}(-)x^i
(a_{ij}'\d_j(u)+c_{ij}'((-)\d_j(u))).
$$
 Conversely each
$\phi\in\End_\Acal(\V )$ can be uniquely written as a
$\Acal$--linear combination of
$$
x^i\otimes\d_j,\qquad  (-)x^i\otimes\d_j,\qquad
x^i\otimes((-)\d_j),\qquad (-)x^i\otimes((-)\d_j)
$$
since, writing $u = \displaystyle{\sum_{0\leq i<n-1}}
u_ix^i+u'_i((-)x^i)$,
\begin{eqnarray*}
\phi(u)&=&\phi\left(\sum_{0\leq i<n-1}
u_ix^i+u'_i((-)x^i)\right)=\sum_{0\leq i<n-1}
u_i\phi(x^i)+u_i'\phi((-)x^i)\cr
&=&\sum_{i=0}^{n-1}u_ia_{ij}x^j+a_{ij}'((-)x^j)+u_i'c_{ij}x^j+u_i'c'_{ij}((-)x^j)\cr
&=&\displaystyle{\sum_{i,j=0}^{n-1}}\left[ x^i\otimes
(a_{ij}\d_j+c_{ij}((-)\d_j))\right.\cr &+& \left (-)x^i\otimes
(a_{ij}'\d_j+c_{ij}'((-)\d_j))\right](u). \hskip150pt \qed
\end{eqnarray*}

Proposition~\ref{action} defines a natural action of  $\gl (\tbw^1\V
)$ on $\V $.
\begin{definition}
The \textbf{Lie bracket} $[\phi,\psi]$ of $\phi,\psi\in \gl(\tbw^1\V
)$ is defined by:

$$
[\phi,\psi]=\phi\circ\psi\ (-)\ \psi\circ\phi.
$$
\end{definition}
\begin{proposition} The $\Acal$-semialgebra $\gl(\tilde T^1(\V ))$ is a  Lie $\preceq$-semialgebra  with a
negation map in the sense of Definition~\ref{Lies}.
\end{proposition}
\proof Items (i),(ii) are obvious, by construction.  For all $v\in\V $
\begin{eqnarray*}
[\phi,\psi](u)\w v&=&\big(\phi(\psi(u))(-)\psi(\phi(u))\big)\w v =
\phi(\psi(u)\w v+v\w \psi(\phi(u)),
\end{eqnarray*}
and
\begin{eqnarray*}
[\psi,\phi](u)\w v&=&\big(\psi(\phi(u))(-)\phi(\psi(u))\big)\w v =
\psi(\phi(u)\w v+v\w \phi(\psi(u)).
\end{eqnarray*}
Thus
$$
[\psi,\phi](u)\w v=\big((-)[\phi,\psi](u)\big)\w v.
$$

 (iii) is a routine verification, done in \cite{Row16}.\qed
\begin{example}\label{ex58}  Let $b\in \V $ and $\d\in \V ^*$. Then $b\otimes\partial$ acts on $v\in\tbw\V $ as:
$$
(b\otimes\d)(v)=b\w (\d\, \lrcorner \, v).
$$
It is a $\Acal$-derivation, i.e. is trivial on $\Acal$ and
$$
(b\otimes\d)(u\w v)=((b\otimes\d)(u))\w v+u\w ((b\otimes\d)(v)) .$$
The critical case to verify is for $u =v_1$ and  $v = v_2\in
\tbw^1\V $, for which we have
$$
(b\otimes\d)(v_1\w v_2)=\d(v_1))b\w v_2+\d(v_2)v_1\w b .$$ The
general case follows by induction.  In fact, if $u\in \tbw^1\V $:
\begin{eqnarray*}
(b\otimes\d)(u\w v)&=&b\w (\d(u)v\,(-)\,u\d(v)) \d(u)b\w v+\d(v)u\w
b\cr &=&\d(u)b\w v+u\w b\d(v)=(b\otimes \d)(u)\w v+u\w (b\otimes
\d)(v).
\end{eqnarray*}
\end{example}

\claim{\bf Extending the $\gl (\tbw^1\V )$-action to $\tbw_\Acal\V
$. }


Consider the map \be \left\{\begin{matrix} \delta&:&\gl(\tbw^1\V
)&\lra& \End_\Acal(\tbw\V )\cr\cr &&\phi&\longmapsto&\delta(\phi)
\end{matrix}\right.
\ee defined by $\delta(\phi)(u)=\phi(u)$ if $u\in \tbw^1\V $ and
inductively
$$
\delta(\phi)(u\w v)=\delta(\phi)(u)\w v+u\w \delta(\phi)(v),\qquad
\forall u\in \gl(\tbw^1\V ),\ v\in \gl(\tbw^{k-1}\V ).
$$

\begin{proposition}
The exterior $\Acal$--semialgebra $\tbw_\Acal\V $ is a module over
the Lie semialgebra $\gl(\bw^1\V )$ (we say in short that it is a
$\gl(\bw^1\V )$-module),  in the sense that \be
\phi(\psi(u))(-)\psi(\phi(u))=[\phi,\psi](u). \ee
\end{proposition}

\proof For all $v\in\V $,
$$
\big(\phi(\psi (u))(-)\psi(\phi(u))\big)\w v=\phi(\psi(u))\w v+v\w
\psi(\phi(u))=[\phi,\psi](u)\w v.
$$
By induction,
$$
\delta([\phi,\psi])=[\delta(\phi),\delta(\psi)]
$$
where the Lie bracket of $\End_\Acal(\tbw\V )$ is defined
analogously.\qed

In the sequel we consider  the action  of $\V \otimes \V ^*$
on $\tbw \V $ given by:
$$
(u\otimes \d )(\bfv)=u\w (\d\, \lrcorner \, \bfv).
$$
which corresponds to the case $B(x^i,x^j)=\delta_{ij}$.
\begin{proposition} The following commutation rules hold for any $u\in
\tbw\V $:
\begin{eqnarray}
&&x^i\w x^j \w u\, +\, x^j\w x^i \w u\succeq
0\label{eq:commbb},\\\cr && \d_i\, \lrcorner \, (\d_j \,
\lrcorner\, u)
 +\, \d_j\, \lrcorner \, ( \d_i\, \lrcorner \, u)
\succeq 0\label{eq:commbebe},\\ \cr && \d_i\, \lrcorner \, (x^j\w u)
+x^i\w(\d_j\,\lrcorner \, u) \,\succeq\label{eq:commbbe} \delta_{ij}
\Bb(x^i,x^j)u.
\end{eqnarray}
\end{proposition}
\proof
 For $u\in \tbw\V $, the first rule
$$
x^i\w x^j \w u+x^j\w x^i\w u\succeq 0
$$
 is immediate. As for $\eqref{eq:commbebe},$
$$
\d_i\, \lrcorner \,(\d_j\, \lrcorner \, u)+\d_j\, \lrcorner
\, (\d_i\, \lrcorner \, u)=(\d_i\w \d_j+\d_j\w
\d^i)\, \lrcorner \, u\succeq 0.
$$
Let us finally check \eqref{eq:commbbe}. For $u\in  \tbw^{ 1}\V $
              we claim that
\be \d^j\lrcorner(x^i\w u)+x^i\w (\d^j\lrcorner u)\succeq
\delta_{ij}\Bb(x^i,x^j)u\label{eq:comp}. \ee Since both sides are
elements of $\tbw^1\V $, \eqref{eq:comp} means that
$$
\left[\d^j\lrcorner (x^i\w u)+x^i\w (\d^j\lrcorner u)\right]\w
v\succeq \delta_{ij}\Bb(x^i,x^j)u\w v
$$
for all $v\in\tbw^{\geq 1}\V $. Without loss of generality we may
assume that $v\in \V $. Then: \be \begin{aligned} \d^j\lrcorner
(x^i\w u)\w v+x^i\w (\d^j\lrcorner u)\w v
=&\delta_{ij}\Bb(x^i,x^j)u\w v+\d^j(u)v\w x^i+\d^j(u)x^i\w v \\
&\succeq \delta_{ij}\Bb(x^i,x^j)u\w v,
\end{aligned}\ee
whence \eqref{eq:commbbe}, since  $v$ is arbitrary.

For $u\in  \bw^{\geq 2}\V $, the proof consists in writing $u=u_1\w
v$ for $u_1\in\V $, and is analogous to that for $u\in \bw^1\V $:
\begin{eqnarray*}
\d^j\, \lrcorner \, (x^i\w u)+x^i\w \d^j\, \lrcorner \,
u&=&\delta_{ij}\Bb(x^i,x^j)u\,(-)\,x^i\w \d^j\, \lrcorner \, (u_1\w
v)+x^i\w\d^j\, \lrcorner \, (u_1\w v)\cr
&=&\delta_{ij}\Bb(x^i,x^j)u+\d^j(u_1)v\w x^i+\d^j(u_1) x^i\w v \cr
&\succeq& \delta_{ij}\Bb(x^i,x^j)\, u.
\end{eqnarray*}

%

%

\section{Schubert Derivations on exterior semialgebras of type
$1$}\label{def:HSD}

\begin{definition}\label{def:defHSD}
A \textbf{Hasse-Schmidt} (HS) derivation on $(\tbw_\Acal V,(-)) $ is
a map
$$
\Dcal(z):\tbw_\Acal\V \sra \tbw_\Acal\V [[z]]
$$
such that
$$
\Dcal(z)(u\w v)=\Dcal(z)u\w \Dcal(z)v.
$$

\end{definition}
\begin{definition}
Let $f\in \End_\Acal(\bw^1\V )$. We denote by $(-)f$ the map such
that
$$
((-)f)(u)=(-)f(u).
$$
\end{definition}
In other words $((-)f)(u)\w v=v\w f(u)$.
Then $\Dcal_f(z)$ and $\ovDc_f(z)$ are mutually inverses, where, according to \ref{def:defnegmap}
$$
(-f)(u)\w v=v\w f(u).
$$
\begin{definition}
Let $\Dcal(z)$ be a HS--derivation on $\tbw_\Acal\V $. Its transpose
$\Dcal(z)^T$ is defined by
$$
\Dcal(z)^T(\d)\lrcorner u=\d\lrcorner \Dcal(z)u\qquad \forall (u,
\d)\in\tbw\V \times \tbw_\Acal\V ^*.
$$
\end{definition}
\begin{proposition}
If $\Dcal(z)$ is a HS--derivation on $\tbw_\Acal\V $, then
$\Dcal(z)^T$ is a HS--derivation on $\tbw_\Acal\V ^*$.
\end{proposition}
\proof $\Dcal(z)^T(\d\w \gamma)\lrcorner (u\w v) = (\d\w
\gamma)\lrcorner \Dcal(z) (u\w v) =  (\d\w \gamma)\lrcorner
(\Dcal(z)u\w \Dcal(z)v ) =\d\lrcorner \Dcal(z)u \w \gamma
\lrcorner \Dcal(z)v = \Dcal(z)^T(\d)\lrcorner u \w
\Dcal(z)^T(\gamma)( v).$\qed
\begin{proposition} [{\cite[Theorem~3.12]{GaR}}]
If $f(z)\in\End_\Acal(\V )[[z]],$ there is a unique HS-derivation
$\Dcal_{f(z)}(z)$ on $\tbw\V $ such that $\Dcal_f(z)_{|\V
}=\sum_{i\geq 0}f^iz^i$.
\end{proposition}

\begin{proposition}\label{prop:prop514}
The unique HS-derivations $\Dcal_f(z)$ and $\ovDc_f(z)$ such that
$$
\Dcal_f(z)=\sum_{n\geq 0}f^nz^n\qquad\mathrm{and}\qquad
\ovDc_f(z)=1\,(-)\, f\cdot z
$$
are  quasi--inverses.

\end{proposition}
\proof Let $(u,v_1,v_2)\in \tbw^1 \V \times \tbw^1 \V \times \tbw\V
$. Then
\begin{eqnarray*}
(\ovDc_f(z)\Dcal_f(z)u)\w (v_1\w
v_2)&=&(\Dcal_f(z)u\,(-)\,\Dcal_f(z) f(u))\w v_1\w v_2\cr &=&
(\Dcal_f(z)u\w v_1+v_1\w \Dcal_f(z)f(u))\w v_2\cr\cr &=&\big(u\w
v_1+z\Dcal_f(z)f(u)\w v_1+v_1\w \Dcal_f(z)f(u)\big)\w v_2\cr\cr
&\succeq& u\w (v_1\w v_2).
\end{eqnarray*}
Similarly
\begin{eqnarray*}
(\Dcal_f(z)\ovDc_f(z)u)\w (v_1\w v_2)&=&[\Dcal_f(z)(u(-)f(u)z)]\w
v_1\w v_2\cr\cr &=&[\Dcal_f(z)u(-)\Dcal_f(z)f(u)]\w v_1\w v_2\cr\cr
&=&(\Dcal_f(z)u\w v_1+v_1\w \Dcal_f(z)f(u))\w v_2\cr
&=&(\Dcal_f(z)u\w v_1+v_1\w \Dcal_f(z)f(u))\w v_2\cr &=&u\w
v_1+z\Dcal_f(z)f(u) \w v_1+v_1\w \Dcal_f(z)f(u))\w v_2\cr &\succeq&
u\w (v_1\w v_2).
\end{eqnarray*}
\qed

\begin{definition}\label{def:defsd}
The \textbf{Schubert derivation} of $\tbw_\Acal\V $ is the unique
HS--derivation  $\sigma_+(z):=\sum_{i\geq 0}\sigma_iz^i\in \End
\,(\tbw_\Acal\V )[[z]]$ satisfying
$$
\sigma_+(z)x^j=\sum_{i\geq 0}x^{j+i}z^i,\qquad (x^{k}=0 \,\,
\mathrm{if} \,\, k\geq n),
$$
and
$$
\sigma_+(z)(u\w v)=\sigma_+(z)u\w \sigma_+(z)v\in\tbw_\Acal\V [[z]].
$$
\end{definition}

\begin{proposition}
For all $i\geq 1$,
$$
\sigma_i(u\w v)=\sum_{j=0}^i\sigma_ju\w\sigma_{i-j}v.
$$
\end{proposition}
\proof \be \begin{aligned} u\w v+\sigma_1(u\w v)z & +\sigma_2(u\w
v)z^2+\cdots=c\\&(u+\sigma_1u\cdot z+\sigma_2u\cdot z^2+\cdots)\w
(v+\sigma_1v\cdot z+\sigma_2v\cdot z^2+\cdots)
\end{aligned}\ee
The result follows by matching the coefficients of the powers of
$z$. \qed

In particular
$
\sigma_1(u\w v)=\sigma_1u\w v+u\w \sigma_1v
$
shows that $\sigma_1$ is an $\Acal$ -derivation.
\begin{example}
$$
\sigma_3(x^3\w x^1)=x^6\w x^1+x^5\w x^2+x^4\w x^3+x^3\w x^4\succeq
x^6\w x^1+x^5\w x^2.
$$
\end{example}
\begin{definition}
Let $z,w$ be two formal variables. The respective generating series
of the bases $\bfx$ and ${\bm\d}$ are
$$
\bfx(z):=\sum_{i\geq 0}x^iz^i\in\V [[z]]\qquad \mathrm{and}\qquad
{\bm\d}(w^{-1})=\sum_{j\geq 0}\d_jw^{-j}\in \V ^*[[w^{-1}]].
$$
\end{definition}
\begin{remark}
Because of definition \ref{def:defsd} we have \be
\bfx(z)=\sigma_+(z)x^0 .\ee Moreover: \be
{\bm\d}(w^{-1})=\sigma_-(w)^T\d_0.\label{eq:for510} \ee In fact
$\sigma_-(z)=\sum_{i\geq 0}\sigma_{-i}w^{-i}$, where
$\sigma_{-i}x^j=x^{j-i}$ if $j\geq i$ and $0$ otherwise. Now, by
definition of transpose:
$$
\sigma_{-i}^T\d_j(x^k)=\d_j(\sigma_{-i}x^k)=\d_j(x^{k-i})=\d_{j+i}(x^k).
$$
Therefore:
$$
\sigma_-(z)^T\d_0=\sum_{j\geq
0}\sigma_{-j}^T\d^0w^{-j}=\sum_{j\geq 0}\d^jw^{-j}.
$$
\end{remark}

\begin{definition}
We denote by $\sigma_-(z), \ovsig_-(z):\widetilde{\bw\V }\sra \bw
\widetilde{ \V }[z^{-1}]$ the unique HS derivations on~$\tbw \V $
such that
$$
\sigma_-(z)x^j=\sum_{i=0}^j{x^{j-i}\over z^i}
$$
and
$$
\ovsig_-(z)x^j=x^j(-){x^{j-1}\over z}
$$
with the convention that $x^i=0$ if $i<0$.
\end{definition}

Notice that $\ovsig_-(z)x^0=\sigma_-(z)x^0=x^0$, i.e. $\ovsig_-(z)$
and $\sigma_-(z)$  act as the identity on $x^0$. By general facts
proved in \cite{GaR} we know that $\sigma_-(z)$ is a quasi-inverse
of $\ovsig_-(z)$, and as defined in this broader context we obtain
the other side. Let us check it again in this particular case. By
Proposition~\ref{prop:prop514}, the Schubert derivations
$\sigma_-(z)$ and $\ovsig_-(z)$ are quasi-inverses of each other. We
offer here an extra check to  focus better the proof of Proposition
\ref{prop:prop514}.
\begin{proposition}\label{prop424}
The Schubert derivations $\ovsig_-(z)$ and $\ovsig_-(z)$ are mutual
quasi-inverses, i.e.,
$$
\sigma_-(z)\ovsig_-(z)\wbx^r_\blamb\succeq \wbx^r_\blamb
$$
and
$$
\ovsig_-(z)\sigma_-(z)\wbx^r_\blamb\succeq \wbx^r_\blamb.
$$

\end{proposition}
\proof Let $v$ be any test element of the form $v_1\w v_2$ with
$v_1\in \V $ and $v_2$ arbitrary in $\tbw\V $. Then
\begin{eqnarray*}
\sigma_-(z)(\ovsig_+(z)x^i)\w v&=&\ovsig_-(z)\left(x^i-{x^{i-1}\over
z}\right)\w v\cr &=&\left(\sigma_-(z)x^i\w v_1+v_1\w {1\over
z}\sigma_-(z)x^{i-1}\right)\w v_2\cr &=&\left(x^i\w v_1+{1\over
z}\sigma_-(z)x^{i-1}\w v_1+\ v_1\w {1\over
z}\sigma_-(z)x^{i-1}\right) \w v_2\cr\cr
&\succeq& x^i\w v_1\w v_2=x^i\w v,
\end{eqnarray*}
which proves the property for $r=1$. If $r\geq 2$
$$
\sigma_-(z)(\ovsig_+(z)\wbx^r_\blamb=\sigma_-(z)\ovsig_-(z)x^{r-1+\lambda_1}\w
\sigma_-(z)\ovsig_-(z)\wbx^{r-1}_{\blamb^{(1)}}\succeq \wbx^r_\blamb,
$$
where $\blamb^{(1)}:=(\lambda_2\geq\ldots\geq \lambda_r)$.
Conversely
\begin{eqnarray*}
\ovsig_-(z)(\sigma_+(z)x^i)\w v_1\w v_2&=&(\sigma_+(z)x^i
(-)z^{-1}\sigma_+(z)x^{i-1})\w v_1\w v_2\cr &=&(x^i\w v_1+
z^{-1}\sigma_-(z) x^{i-1}\w v_1+v_1\w z^{-1}\sigma_+(z)x^{i-1})\w
v_2\cr &\succeq& x^i\w v_1\w v_2=x^i\w v.\hskip 160pt \qed
\end{eqnarray*}

\begin{theorem}\label{thm:thm313}
We have \be \sigma_+(z)x^0\w \wb^r_\blamb\preceq
z^r\sigma_+(z)\ovsig_-(z)\wbx^{r+1}_\blamb\label{eq5:510} \ee for all
$\blamb\in\Pcal_r$ such that $\lambda_1 \le n-r$.
\end{theorem}
\proof
We argue by induction. We first check the claim for $r=1$, i.e., by analyzing
$$
\sigma_+(z)x^0\w x^{\lambda},\qquad \lambda<n.
$$
We have
\begin{eqnarray*}
z\sigma_+(z)\ovsig_-(z)( x^{\lambda+1}\w
x^0)&=&z\sigma_+(z)(\ovsig_-(z)x^{\lambda+1}\w \ovsig_-(z)x^0)\cr\cr
&=&z\sigma_+(z)\left[\left(x^{\lambda+1}(-){x^\lambda\over z}\w
x^0\right)\right]\cr\cr &=&z\sigma_+(z)\left(x^{\lambda+1}\w
x^0+{1\over z} x^0\w x^\lambda\right)\cr\cr
&=&z\sigma_+(z)x^{\lambda+1}\w\sigma_+(z)x^0+\sigma_+(z)x^0\w
\sigma_+(z)x^\lambda\cr\cr &=&z\sigma_+(z)x^{\lambda+1}\w
\sigma_+(z)x^0+(\sigma_+(z)x^0\w\sigma_+(z)x^{\lambda+1})\cr\cr &=&
z\sigma_+(z)x^{\lambda+1}\w \sigma_+(z)x^0+ \sigma_+(z)x^0\w
x^\lambda+z\sigma_+(z)x^0\w\sigma_+(z)x^{\lambda+1}\cr\cr &\succeq&
 \sigma_+(z)x^0\w x^\lambda,
\end{eqnarray*}
and the property is verified for $r=1$. Assume now that $r-1\geq 1$.
Then:
\begin{eqnarray*}
&&z^r\sigma_+(z)\ovsig_-(z)\wbx^{r+1}_\blamb\cr\cr &=&z^{r-1}(\sigma_+(z)\ovsig_-(z)x^{r+\lambda_1}\w\cdots\w
x^{2+\lambda_{r-1}})\w z\sigma_+(z)\ovsig_-(z)(x^{1+\lambda_r}\w
x^0)\cr\cr &\succeq&
z^{r-1}\sigma_+(z)\ovsig_-(z)(x^{r+\lambda_1}\w\cdots\w
x^{2+\lambda_{r-1}})\w\sigma_+(z)x^0\w x^{\lambda_r}\cr\cr
&=&z^{r-1}\sigma_+(z)\ovsig_-(z)(x^{r+\lambda_1}\w\cdots\
x^{2+\lambda_{r-1}}\w x^0)\w x^{\lambda_r}\cr\cr &\succeq&
[\sigma_+(z)x^0\w x^{r-1+\lambda_1}\w\cdots\w x^{1+\lambda_{r-1}}]\w
x^{\lambda_r}=\sigma_+(z)x^0\w \wbx^r_\blamb.
\end{eqnarray*}
\qed
%

\begin{definition}{ The Lie semialgebra $\gl(\tbw^1_\Acal\Acal[x])\cong \tbw^1\V \otimes \tbw^1\V ^*$ acts on $\tbw\V $ as follows:
$$
(p\otimes\d)(u)=p\w (\d\lrcorner u).
$$
}
\end{definition}
A standard procedure in representation theory is using generating functions.
\begin{definition}\label{def:def716} Let $\bfx:=\{x^i : i\geq 0\}$ and ${\bm\d}:=(\{\d_j :  j\geq 0\}$ be the bases of $\V $ and $\V ^*$. The generating functions of $\bfx$ and ${\bm\d}$
are by definition:
$$
\bfx(z)=\sum_{i\geq 0}x^iz^i\qquad \mathrm{and}\qquad
{\bm\d}(w^{-1})=\sum_{i\geq 0}\d_jw^{-j}.
$$
\end{definition}
In particular $\bfx(z)\otimes {\bm\d}(w^{-1})$ defines a map
$\tbw^k_\Acal\V \sra \tbw^k_\Acal\V [[z, w^{-1}]]$. The image of a
basis element is prescribed as follows:
\begin{eqnarray}
(\bfx(z)\otimes {\bm\d}(w^{-1}))  \wbx^k_\blamb &:=&\bfx(z)\w ({\bm
\d}(w^{-1})\, \lrcorner \, \wbx^k_\blamb.
\end{eqnarray}

%
By definition
${\bm\d}(w^{-1})=\d^0+\d^1w^{-1}+\d^2w^{-2}+\cdots$.
So ${\bm\d}(w^{-1})\lrcorner x^i=w^{-i}$ for all $0\leq i<n$.
Now we apply Definition \ref{def:conr} of  contraction, obtaining
$$
\bfx(z)\w ({\bm \d}(w^{-1})\lrcorner
\wbx^k_\blamb=\bfx(z)\w
\begin{vmatrix}w^{-k+1-\lambda_1}&\cdots& w^{-\lambda_k}\cr\cr
x^{k-1+\lambda_1}&\cdots&x^{\lambda_k}\end{vmatrix}=\sigma_+(z)x^0\w
\begin{vmatrix}w^{-k+1-\lambda_1}&\cdots& w^{-\lambda_k}\cr\cr
x^{k-1+\lambda_1}&\cdots&x^{\lambda_k}\end{vmatrix},
$$
where in the last equality we used \ref{def:defsd}. \qed

\begin{example}
The image of $\wbx^3_{(3,2,1)}$ through the endomorphism $x^2\otimes
\d_3$  of $\V $ is easily computable by hand:
\begin{eqnarray*}
(x^2\otimes \d_3)\bfx^3_{(3,2,1)}=x^2\w (\d_3\lrcorner (x^5\w
x^3\w x^1))&=&x^2\w
\begin{vmatrix}\d_3(x^5)&\d_3(x^3)&\d_3(x^1)\cr\cr x^5&x^3&x^1
\end{vmatrix}\cr\cr\cr &=&x^2\w \begin{vmatrix}0&1&0\cr\cr
x^5&x^3&x^1\end{vmatrix}=x^2\w  x^1\w x^5.
\end{eqnarray*}
On the other hand $(x^2\otimes \d_3)\wbx^3_{(3,2,1)}$ should be
the coefficient of $z^2w^{-3}$ of the expansion of \be
\sigma_+(z)x^0\w
\begin{vmatrix} w^{-5}&w^{-3}&w^{-1}\cr x^5&x^3&x^1
\end{vmatrix}.\label{eq:gfex}
\ee
The coefficient of $w^{-3}$ in expression \eqref{eq:gfex}
$$
\sigma_+(z)x^0\w (x^1\w x^5)
$$
whose coefficient of $z^2$ is precisely $x^2\w x^1\w x^5$, as
expected.
%
\end{example}

\begin{remark}
Let us consider the standard bilinear form $B(x^i,x^j)=\delta_{ij}$. Recall by formula \eqref{eq:commbbe} that
$$
x^i\w \d^j\lrcorner \wbx^r_\blamb\succeq \delta_{ij}\wbx^r_\blamb
$$
which are Clifford semialgebra relations. Passing to the generating
series:
$$
\sum_{i,j\geq 0}x^iz^i\w \d^jw^{-j}\lrcorner \wbx^r_\blamb\succeq \delta_{ij}z^iw^{-j}\wbx^r_\blamb
$$
from which, remembering  Definition \ref{def:def716} for ${\bm\d}(w^{-1})$ and Definition \ref{def:defsd} for the Schubert derivation:
$$
\bfx(z)\w {\bm\d}(w^{-1})\lrcorner\wbx^r_{\blamb}\, +\,
{\bm\d}(w^{-1})\lrcorner(\bfx(z)\w \wbx^r_\blamb)\succeq \sum_{i\geq 0}{z^i\over w^i}\wbx^r_\blamb=\left(1-{z\over w}\right)^{-1}\wbx^r_\blamb
$$
a relation providing  the connection of the commutation rules of Schubert derivations  to Clifford semi Algebras.
\end{remark}

\begin{remark}
If $\Acal$ contains the positive rational numbers then
$$
\sigma_+(z)=\exp\left(\sum_{i\geq 1}{1\over
i}\delta(x^i)z^i\right),
$$
where left multiplication $x\cdot:\V \sra \V $ is   defined by
$x^i\mapsto x^{i+1}$ if $i< n-1$ and $x^{n-1}\mapsto 0$, and $\delta
:\End_\Acal(\V )\sra \End_\Acal (\bw\V )$ is defined as in
Example~\ref{ex58}.
\end{remark}

 \begin{theorem}\label{thm:thm515}  We use \eqref{eq:for510}. Then
$$
\sigma_-(w)^T\d_0\, \lrcorner \,
\wbx^r_\blamb=w^{-r+1}\ovsig_+(w)(\d_0\, \lrcorner \,
\sigma_-(w)\wbx^r_\blamb)
$$

\end{theorem}
\proof

\begin{eqnarray*}
\sigma_-(w)^T\d_0\, \lrcorner \, \wb^r_\blamb&=&{\bm\d}(w^{-1})\,
\lrcorner \,
\wb^r_\blamb=\begin{vmatrix}\d^0(\sigma_-(w)x^{r-1+\lambda_1})&\cdots&\d^0(\sigma_-(w)x^{\lambda_r})\cr\cr\cr
x^{r-1+\lambda_1}&\cdots&x^{\lambda_r}\end{vmatrix}\cr\cr\cr
&=&\begin{vmatrix}\d_0(\sigma_-(w)x^{r-1+\lambda_1})&\cdots&\d^0(\sigma_-(w)x^{\lambda_r})\cr\cr\cr
x^{r-1+\lambda_1}&\cdots&x^{\lambda_r}\end{vmatrix}\cr\cr\cr
&=&\ovsig_-(w)\begin{vmatrix}\d^0(\sigma_-(w)x^{r-1+\lambda_1})&\cdots&\d^0(\sigma_-(w)x^{\lambda_r})\cr\cr
\sigma_-(w)x^{r-1+\lambda_1}&\cdots&\sigma_-(w)x^{\lambda_r}\end{vmatrix}=\cr\cr\cr
&=&{1\over w^{r-1}}\ovsig_+(w)\ovsig_-(w)(\d^0\, \lrcorner \,
\sigma_-(w)[\bfb]^r_\blamb)\cr
\end{eqnarray*}

\claim{\bf The  explicit description of the $\tbw_\Acal\V $
representation of $\gl(\tilde T^1(\V ))$.} Theorem \ref{thm432}
below gives a first version of the structure of $\tbw\V $ as a
representation of $\gl(\tbw^1\V )$. To this purpose we need a few
preliminaries.
\begin{lemma}\label{lem633} For all $\blamb\in \Pcal_r$,
\be x^0\w [\bfb]^r_{\blamb+(1^r)}\preceq{1\over
w^r}\ovsig_+(w)\sigma_-(w)\wbx^{r}_\blamb\w x^0 \ee

\end{lemma}
\proof
For $r=1$ one has
\begin{eqnarray*}
x^0\w \wbx^1_{(1+\lambda)}&=&x^0\w x^{1+\lambda}\preceq
x^0\w\ovsig_-(w)\sigma_-(w)x^{1+\lambda}=x^0\w
\ovsig_-(w)\left(\sum_{j=0}^{1+\lambda}{x^{1+\lambda-j}\over
w^j}\right)\cr &=&x^0\w \sum_{j=0}^{1+\lambda}{1\over
w^j}\ovsig_-(w)x^{1+\lambda-j}=x^0\w \sum_{j=0}^{1+\lambda}{1\over
w^j}\left(x^{1+\lambda-j}(-){x^{\lambda-j}\over w}\right)\cr\cr
&=&{1\over w}\left(\sum_{j=0}^\lambda{1\over
w^j}\ovsig_+(w)x^{\lambda-j}+{x^0\over w^{\lambda+1}}\right)\w
x^0\cr\cr &=&{1\over
w}\left(\ovsig_+(w)\sigma_-(w)x^\lambda+{x^0\over
w^{\lambda+1}}\right)\w x^0={1\over
w}(\ovsig_+(w)\sigma_-(w)x^\lambda)\w x^0
\end{eqnarray*}
as desired. For $r>1$ we argue by induction. Suppose the formula
holds for $r-1\geq 1$. Then
\begin{eqnarray*}
x^0\w \wbx^r_{\blamb+(1^r)}&=&x^0\w x^{r+\lambda_1}\w
x^{r-1+\lambda_2}\w\cdots\w x^{1+\lambda_r}\cr&\preceq&{1\over w
}\ovsig_+(w)\sigma_-(w)x^{r-1+\lambda_1}\w x^0 \w
\wbx^{r-1}_{(r-1+\lambda_2,\ldots,1+\lambda_r)}\cr\cr &=&{1\over
w^r}\ovsig_+(w)\sigma_-(w) x^{r-1+\lambda_1}\w
\ovsig_+(w)\sigma_-(w)(x^{r-2+\lambda_2}\w\cdots\w x^{\lambda_r})\w
x^0\cr\cr &=&{1\over w^r}\ovsig_+(w)\sigma_-(w)\wbx^r_\blamb\w
x^0.
\end{eqnarray*}

\qed
\begin{lemma}\label{lem:lemma433}
The following commutation rule holds:
$$
\ovsig_-(z)\ovsig_+(w)[\wbx]^r_\blamb\w
x^0=\ovsig_+(w)\ovsig_-(z)\wbx^r_\blamb\w x^0
$$
\end{lemma}
\proof
We first prove the case $r=1$. If $\lambda>0$ one has
\begin{eqnarray*}
\ovsig_-(z)\ovsig_+(w)x^\lambda=\ovsig_-(z)(x^\lambda(-)x^{\lambda+1}z)&=&\sigma_-(w)\ovsig_+(z)\wbx^r_\blamb\w
x^0\cr\cr &=&\left(1+{z\over w}\right)x^\lambda
(-)\left(x^{\lambda+1}+{1\over w}x^{\lambda-1}\right)\cr\cr
&=&\ovsig_+(w)\ovsig_-(z)x^\lambda
\end{eqnarray*}
from which it is obvious that
$\ovsig_+(z)\sigma_-(w)x^\lambda\w x^0=\sigma_-(w)\ovsig_+(z)x^\lambda\w x^0$. It remains to check the commutations of the two oprators against $x^0$.
In this case
$$
\ovsig_+(z)\ovsig_-(w)x^0\w x^0=x^0\w x^1z=\ovsig_-(w)\ovsig_+(z)x^0\w x^0
$$
as a straightforward check shows. It remains to check the general
case for which we argue by induction. Suppose the property holds for
$r-1\geq 0$. Then
\begin{eqnarray*}
&&\ovsig_-(z)\ovsig_+(w)\wbx^r_\blamb\w x^0\cr\cr
&=&\ovsig_-(z)\ovsig_+(w)x^{r-1+\lambda_1}\w\cdots\w
\ovsig_-(z)\ovsig_+(w)x^{1+\lambda_{r-1}}\w
(\ovsig_-(z)\ovsig_+(w)x^{\lambda_r}\w x^0)\cr\cr
&=&\ovsig_-(z)\ovsig_+(w)x^{r-1+\lambda_1}\w\cdots\w
\ovsig_-(z)\ovsig_+(w)x^{1+\lambda_{r-1}}\w
(\ovsig_+(w)\ovsig_-(z)x^{\lambda_r}\w x^0)\cr\cr
&=&(-)\ovsig_-(z)\ovsig_+(w)x^{r-1+\lambda_1}\w\cdots\w
\ovsig_-(z)\ovsig_+(w)x^{1+\lambda_{r-1}}\w x^0)
\w\ovsig_+(w)\ovsig_-(z) x^{\lambda_r}
\end{eqnarray*}
from which, by induction
\begin{eqnarray*}
&(-)&(\ovsig_+(w)\ovsig_-(z)x^{r-1+\lambda_1}\w\cdots\w
\ovsig_+(w)\ovsig_-(z)x^{1+\lambda_{r-1}}\w x^0)\w
\ovsig_+(w)\ovsig_-(z) x^{\lambda_r}\cr\cr
&=&\ovsig_+(w)\ovsig_-(z)x^{r-1+\lambda_1}\w\cdots\w
\ovsig_+(w)\ovsig_-(z)x^{1+\lambda_{r-1}}\w \ovsig_+(w)\ovsig_-(z)
x^{\lambda_r}\w x^0\cr\cr &=&\ovsig_+(w)\ovsig_-(z)\wbx^r_\blamb\w
x^0\hskip250pt \qed
\end{eqnarray*}

\section{The Main Theorem}\label{8mnthm}

In this section we will present and prove  our main theorem after relating it to previous  literature in the classical context of $\QQ$-algebras. To
begin with, a rather easy, though non trivial,  observation shows that the polynomial
ring $B:=\QQ[x_1,x_2,\ldots]$ in infinitely many indeterminates $\bfx=(x_1, x_2,\ldots)$ is a module over the
Lie algebra $gl_\infty(\QQ)=\bigoplus_{i,j\in\ZZ} \QQ\cdot E_{ij}$, where $E_{ij}:\ZZ\times \ZZ\sra \QQ$ is
the elementary matrix with all entries zero but $1$ in position $(i,j)$. This fact
comes from the fact that the ring $B$ is isomorphic to a suitable projective  limit, $\underset{r}\varprojlim \bw^r\QQ[X]$, in the category of graded $\QQ$-algebras, called {\em Fermionic Fock space}. The
 result then basically follows from the fact that a) the polynomial ring in $r<\infty$
variables is a vector space isomorphic to the $r$-th exterior power of $\QQ[X]$ and
b) the latter is a module over $gl(\QQ[X]):=\bigoplus_{i,j\in\NN} \QQ\cdot E_{ij}$.

It is more difficult to give an explicit description of $B$ as a
representation of $gl_ \infty(\QQ)$. This was achieved in the 1980's
by Date, Jimbo, Kashiwara and Miwa in \cite{DJKM}, who computed the
action of the generating function $E(z,w)=
\sum_{i,j\in\ZZ}E_{i,j}z^iw^{-j}$ on a polynomial in the variables
$x_1,x_2,\ldots$, in such a way that to know the product $E_{ij}p$
  amounts to looking at the coefficient of $z^iw^{-j}$ in the expansion
of $E(z,w)p$, for all $p\in B$. The description relies on some
vertex operators which first arose in the Skyrme model \cite{skyrme}
of the interaction of a meson--like field. The result by Date,
Jimbo, Kashiwara and Miwa was put in a more general framework in the
paper \cite{GaSa2}, where a generating function of the action of the
elementary matrices on the ring $B_r$ of polynomials in $r<\infty$
indeterminates has been computed. In our case partial derivatives
with respect to the variables $(x_1,x_2,\ldots)$ are no longer
available anymore for describing the module structure. This is an
issue which has compelled us to find suitable substitutes. These are
provided by the essential use of the Schubert derivations as defined
in  e.g. \cite{SDIWP} and in Section \ref{def:defsd} in the context
of exterior semialgebras. The elementary matrix $E_{i,j}$ can be
identified, using our notation, with the basis element $x^i\otimes
\d^j$ of $V\otimes V^*$, and then the generating function $E(z,w)$
is nothing but $ \bfx(z)\otimes {\bm \d}(w^{-1})$, acting on
$\bw^rV$ according to the rule
$$
\bfx(z)\otimes {\bm\d}(w^{-1}) \wbx^r_\blamb=\bfx(z)\w {\bm\d}(w^{-1})\lrcorner
\wbx^r_\blamb.
$$
The last member can be computed as $\sigma_+(z)x^0\w {\bm
\d}(w^{-1}) \w\wbx^r_\blamb$, and this was done in \cite{GaSa2}
(with substantial improvement in the preprint \cite{BCM}), where the
generating function of the Clifford algebra basis element
$x^i\otimes \d^j$ against $\wbx^r_\blamb$, $x^i\w (\d^j\lrcorner
\wbx^r_\blamb)$ was computed. The resulting expression explicitly
involves
  Schubert derivations which, when $r$ goes to $
\infty$, take the shape of the vertex operators  occurring in the
DJKM  representation of $gl_\infty(\QQ)$,  as   shown e.g.~in
\cite{SDIWP}. Our main Theorem \ref{thm432} shows that the same
results hold in the semialgebra context and then it is a substantial
generalization of \cite{GaSa2}, which motivates further
investigation to check if even the more general pictures as in
\cite{BeCoGaVi,BeGa} can be reproduced in the more general  systemic
framework.

Therefore our main result  is the following transparent tropical
version of \cite[Theorem 6.4]{GaSa2} for the Grassmann semialgebra,
describing precisely the  multiplication of an elementary matrix by
a basis element of $\tbw_\Acal\V$.

\begin{theorem}\label{thm432}  For all $r\geq 1$ the following formula
holds:
\be\begin{aligned} \bfx(z)&\w {\bm\d}(w^{-1})\lrcorner
\wbx^r_\blamb=\sigma_+(z)x^0\w (\sigma_-(w)^T\d^0\, \lrcorner \,
\wbx^r_\blamb) \\& \preceq z^{r-1}w^{r-1}\begin{vmatrix}
w^{-r+1-\lambda_1}&\cdots&w^{-\lambda_r}&0\cr\cr
\sigma_+(z)\ovsig_+(w)\ovsig_-(z)\sigma_-(w)x^{r+
\lambda_1}&\cdots&\sigma_+(z)\ovsig_+(w)\ovsig_-(z)\sigma_-(w)x^{1+
\lambda_r}&\sigma_
+(z)x^0
\end{vmatrix} \\ \\
&= {z^{r-1}\over w^{r-1}}\sigma_+(z)\begin{vmatrix}
w^{-r+1-\lambda_1}&\cdots&w^{-\lambda_r}&0\cr\cr
\ovsig_+(w)\ovsig_-(z)\sigma_-(w)x^{r+
\lambda_1}&\cdots&\ovsig_+(w)\ovsig_-(z)\sigma_-(w)x^{1+\lambda_r}
&x^0 \end{vmatrix}.
\cr\cr && \label{eq:prevo} \end{aligned}\ee
\end{theorem}
The formula means that the image of the element $\wbx^r_\blamb$
through $x^i\otimes\d^j$ is surpassed by the coefficient of
$z^iw^{-j}$ in expression \eqref{eq:prevo}.

\proof We have
$$
 \bfx(z)\w {\bm\d}(w^{-1})\lrcorner
\wbx^r_\blamb
$$

\medskip
\begin{tabular}{rlll}
&$=$&$\sigma_+(z)x^0\w
\begin{vmatrix}w^{-r+1-\lambda_1}&\cdots&w^{-\lambda_r}\cr\cr
x^{r-1+\lambda_1}&\cdots&x^{\lambda_r}
\end{vmatrix}$& (definition of $\sigma_+(z)$ and
\cr
&&& Proposition \ref{prop:prop514})\cr\cr
&$\preceq$&$z^{r-1}\sigma_+(z)\ovsig_-(z)\begin{vmatrix}w^{-r+1-\lambda_1}&
\cdots&w^{-\lambda_r}&0\cr\cr
x^{r+\lambda_1}&\cdots&x^{1+\lambda_r}&x^0
\end{vmatrix}$&(Theorem \ref{thm:thm313})
\end{tabular}

\begin{tabular}{rlll}
&$\preceq$&
$z^{r-1}\sigma_+(z)\ovsig_-(z)\begin{vmatrix}w^{-r-1-\lambda_1}&\cdots&w^{-\lambda_r}
&0\cr\cr
\ovsig_-(w)\sigma_-(w)
x^{r+\lambda_1}&\cdots&\ovsig_-(w)\sigma_-(w)x^{1+\lambda_r}&x^0
\end{vmatrix}$&(Proposition \ref{prop424})
\end{tabular}

\begin{tabular}{rlll}
&$\preceq$&$\displaystyle{z^{r-1}\over w^{-r+1}}\sigma_+(z)\begin{vmatrix}w^{-r-1-
\lambda_1}&\cdots&w^{-\lambda_r}&0\cr\cr
\ovsig_-(z)\ovsig_+(w)\sigma_-(w)
x^{r+\lambda_1}&\cdots&\ovsig_-(z)\ovsig_+(w)\sigma_-(w)x^{1+\lambda_r}&x^0
\end{vmatrix}$&(Theorem \ref{thm:thm515})\cr\cr
&$=$&$\displaystyle{z^{r-1}\over w^{-r+1}}\sigma_+(z)\begin{vmatrix}w^{-r-1-
\lambda_1}&\cdots&w^{-\lambda_r}&0\cr\cr
\ovsig_+(w)\ovsig_-(z)\sigma_-(w)
x^{r+\lambda_1}&\cdots&\ovsig_+(w)\ovsig_-(w)\sigma_-(w)x^{1+\lambda_r}&x^0
\end{vmatrix}$&(Lemma \ref{lem633})
\end{tabular}

\bigskip
or, using the fact that $\sigma_+(z)$ is a HS derivation
$$
\displaystyle{z^{r-1}\over w^{-r+1}}\begin{vmatrix}w^{-r-1-\lambda_1}&\cdots&w^{-
\lambda_r}&0\cr\cr
\sigma+(z)\ovsig_+(w)\ovsig_-(z)\sigma_-(w)
x^{r+\lambda_1}&\cdots&\sigma_+(z)\ovsig_+(w)\ovsig_-(w)\sigma_-(w)x^{1+\lambda_r}&
\sigma_+(z)x^0
\end{vmatrix}
$$
\qed

\end{document}